\documentclass[leqno,oneside]{article}

\usepackage[left=4.5cm,right=4.5cm,bottom = 3.9cm]{geometry}
\usepackage{graphics,graphicx}
\usepackage{subfigure}
\usepackage{psfrag}
\usepackage{enumerate}
\usepackage{enumitem}
\usepackage{fancyvrb}
\usepackage{amsmath,amsfonts,amssymb,amsthm}
\usepackage{esint}
\usepackage{exscale} 
\usepackage{tikz}
\usepackage{array}
\usepackage{verbatim}
\usepackage{subfigure}
\usepackage{dsfont}
\usepackage{color}
\usepackage{mathtools}
\usepackage{graphicx}
\usepackage{placeins}
\definecolor{marin}{rgb}{0.,0.3,0.7} 
\definecolor{rouge}{rgb}{0.8,0.,0.} 
\definecolor{sepia}{rgb}{0.8,0.5,0.} 
\definecolor{vert}{rgb}{0.,.55,0.}
\usepackage[colorlinks,citecolor=marin,linkcolor=rouge,
            bookmarksopen,
            bookmarksnumbered
           ]{hyperref}
\usepackage{authblk}
\usepackage{xspace}
\usetikzlibrary{arrows,matrix}
\usepackage{fourier}

\usepackage{fancyhdr}
\usepackage[mathcal]{eucal}

\newtheorem{thm}{Theorem}

\newtheorem{hyp}{Hypothesis}
\newtheorem{lem}[thm]{Lemma}
\newtheorem{prop}[thm]{Proposition}
\newtheorem{rmk}{Remark}
\newtheorem{defi}{Definition}

\def\secref{\S\ref}
\def\hypref{Hypothesis~\ref}
\def\lemref{Lemma~\ref}
\def\defref{Definition~\ref}
\def\theoref{Theorem~\ref}
\def\propref{Proposition~\ref}
\def\rmkref{Remark~\ref}
\def\figref{Figure~\ref}

\makeatletter
\def\blfootnote{\gdef\@thefnmark{}\@footnotetext}
\makeatother

\author[1]{Sebastian Noelle\thanks{noelle@igpm.rwth-aachen.de}}
\author[2]{Martin Parisot\thanks{martin.parisot@inria.fr}}
\author[3]{Tabea Tscherpel\thanks{ttscherpel@math.uni-bielefeld.de}}
\affil[1]{Institute for Geometry and Applied Mathematics, RWTH Aachen, 52056 Aachen, Germany}
\affil[2]{INRIA, Univ. Bordeaux, CNRS, Bordeaux INP, IMB, UMR 5251, 200 Avenue de la Vieille Tour, 33405 Talence cedex, France}
\affil[3]{Faculty of Mathematics, Bielefeld University, 33501 Bielefeld, Germany}
\def\d{{\, \mathrm{d}}}
\def\ds{\displaystyle}
\def\l{\left}
\def\r{\right}
\def\({\l(}
\def\){\r)}
\def\ll{\left\langle}
\def\rr{\reft\rangle}

\def\style{\it}

\newcommand\mat[2]{\begin{array}{@{}#1@{}}#2\end{array}}

\def\u{{\overline u}}
\def\w{{\overline w}}
\def\q{{\overline q}}

\def\P{\mathcal P}
\def\Q{\mathcal Q}
\def\Rest{\mathcal R}

\def\E{\mathcal E}
\def\Gsw{\mathcal G}
\def\Ggn{\mathcal H}

\def\abs#1{\left| #1 \right|}
\newcommand{\norm}[1]{\left|\!\left| #1 \right|\!\right|}
\def\ll{\left\langle}
\def\rr{\right\rangle}
\def\dx{\,\mathrm{d}x}

\def\diver{\nabla \cdot}
\def\divergence{\mathrm{div}}
\def\dtn{\delta_t^n}
\def\dt{\delta_t}

\def\R{\mathbb{R}}
\def\Eset{\mathbb{E}}
\def\I{\mathbf I}

\def\LRh{L^2\left(\R^d;h\right)}
\def\Lh{L^2\left(\Omega;h\right)}
\def\Ahb{\mathbb{A}_h}
\def\Ahgu{\mathbb{A}_{h,\Gamma_u}}
\def\Ahg{\mathbb{A}_{h,\Gamma}}

\def\mAhgu{\mathcal{A}_{h,\Gamma_u}}
\def\mAhguu{\mathcal{A}_{h,\Gamma_u}\left(\bcu\right)}

\def\Hdiv{H(\divergence;\Omega)}

\def\mHghq{\mathcal{H}_{\Gamma_{hq}}}
\def\Gsim{\stackrel{\Gamma}{\equiv}}
\def\Gusim{\stackrel{\Gamma_u}{\equiv}}

\def\Ur{U^r}
\def\ur{\u^r}
\def\qr{\q^r}

\def\Pih{\Pi_h[\Ahb]}
\def\Pihg{\Pi_h[\Ahg]}
\def\Gc{{\Gamma^{\mathrm{c}}}}

\def\bcu{\widetilde{u}}	
\def\bchq{\widetilde{hq}} 	
\def\bch{\widetilde{h}}	
\def\bchu{\widetilde{hu}}
\def\bcv{\widetilde{v}}	
\def\bcw{\widetilde{w}}	

\def\T{\mathbb{T}}
\def\W{\mathbb{W}}
\def\G{\mathbb{G}}
\def\Fd{\mathbb{F}}
\def\F{\mathcal{F}}
\def\Sd{\mathcal{S}}
\def\Ahd{\mathbb A_{h_\star}^\delta}
\def\Ahdg{\mathbb A_{h_\star,\Gamma}^\delta}
\def\m{m}
\def\Lhd{\ell^2\(\T;h_\star\)}
\def\kf{{k_f}}
\def\kg{{k_g}}
\def\ki{{k_i}}
\def\n{\nu}
\def\t{\tau}
\def\nablad{\nabla^\delta}
\def\Gammad{\Gamma^\delta} 

\def\dof{\mathbb{D}}
\def\dofh{\dof^h}
\def\dofu{\dof^{\u}}
\def\dofw{\dof^{\w}}
\def\dofsig{\dof^{\sigma}}
\def\dofa{\dof^{a}}
\def\dofU{\dof^{U}}

\def\Ccfl{C_{\textrm{cfl}}}
\def\sol{\textrm{sol}}
\def\hsol{h_{\sol}}
\def\usol{\u_{\sol}}
\def\wsol{\w_{\sol}}
\def\ssol{\sigma_{\sol}}
\def\qsol{\q_{\sol}}
\def\qBsol{q_{B,\sol}}
\def\phisol{\phi_{\sol}}

\DeclareMathOperator\card{card}
\DeclareMathOperator\Img{Im}
\DeclareMathOperator\sign{sign}
\DeclareMathOperator\sech{sech}

\title{
A class of boundary conditions for time-discrete Green--Naghdi equations with bathymetry
}

\begin{document}

\pagestyle{myheadings}
\markboth{\hfill\sc Sebastian Noelle, Martin Parisot, Tabea Tscherpel\hfill}{
Boundary conditions for time-discrete Green--Naghdi  equations
}

\maketitle

\begin{abstract}
This work is devoted to the structure of the time-discrete Green--Naghdi equations 
including bathymetry. 
We use the projection structure of the equations to characterize homogeneous and inhomogeneous boundary conditions for which the semi-discrete equations are well-posed.
This structure allows us to propose efficient and robust numerical treatment of the boundary conditions that ensures entropy stability of the scheme by construction. 
Numerical evidence is provided to illustrate that our approach is suitable for situations of practical interest that are not covered by existing theory. 
\end{abstract}

\blfootnote{\textup{2020} \textit{Mathematics Subject Classification}:
35F60, 65M08, 65M12,76B15, 76M12.}


\section{Introduction}\label{sec:intro}

The Green--Naghdi model~\cite{GN.1976,S.1953} is a reduced model for free surface flows that is well adapted to the propagation of waves, especially in coastal areas~\cite{WKGS.1995}.
It can be derived from the incompressible free-surface Euler equations, also referred to as water waves model, either by assuming an irrotational flow~\cite{La.2013} or by vertical averaging~\cite{FPPS.2018}. 
Since the Green--Naghdi model is nonlinear and dispersive, the analysis of non-trivial boundary conditions is rather challenging. 
Only few contributions on boundary conditions for the Green--Naghdi equations are available, while the articles~\cite{A.2012,LW.2020} propose some analysis in a similar context.
The fact that in many cases boundary conditions for numerical schemes are tailored to reproduce a specific phenomenon and justified only afterwards shows that we are still far away from a full understanding of boundary conditions for the Green--Naghdi and the water waves model. 
It is the purpose of this work to shed some more light on this topic. 
More specifically, we propose a class of boundary conditions for the time-discrete Green--Naghdi model for which the resulting scheme is entropy-satisfying by construction. 
\medskip

For hyperbolic models, such as the shallow water equations, there is a number of suitable boundary conditions one may pose, depending on the number of characteristics entering the domain.
In this manner standard boundary conditions such as periodic, transparent, symmetric or fixing some of the unknowns have been analyzed in the literature~\cite{HPT.2015,HT.2014b,HT.2014a,HT.2015,PankratzNatvigGjevikNoelle2007,PT.2011,PT.2013}. 
However, both the Green--Naghdi equations and the water waves model are not hyperbolic and for dispersive models like them there is no equivalent to characteristics. 
Usually the analysis is performed on the whole spatial domain, or with periodic or symmetric (wall) boundary conditions, see~\cite{I.2011, L.2005, La.2013, Lannes08, L.2006}.
The strategy for linear dispersive equations proposed in~\cite{A.2012} shares similarities with the approach based on characteristics for hyperbolic problems.
However, it leads to expensive computations that are difficult to perform for the Green--Naghdi model, both for the linearized model and the non-linear one. 

For practical numerical applications on the Green--Naghdi equations mostly periodic or symmetric boundary conditions have been investigated in depth~\cite{SainteMarie16, BCLMT.2011, CLM.2011, Gavrilyuk17, KR.2018}.  
In applications in oceanography a transparent boundary condition used for outgoing waves is indispensable. 
Usually it is replaced by an absorbing boundary layer by adding a source term to dissipate the energy of the wave, see~\cite{KDS.2014}. 
To the best of our knowledge only few contributions go beyond this. 
In~\cite{KN.2020} the authors propose a fine numerical analysis of transparent boundary conditions based on the Dirichlet-to-Neumann map.
Unfortunately this strategy is non-local in time, which makes it quite complex in practice.
A range of recent schemes for the Green--Naghdi model~\cite{SainteMarie16,BCLMT.2011,Gavrilyuk17,Parisot19,Popinet20} apply a prediction-correction strategy, which is well-known for the Euler and Navier--Stokes equations and dates back to~\cite{C.1968, T.1969}, see~\cite{Guermond06} for a review. 
In~\cite{ABGMPS.2020}, a set of boundary conditions is used that mimics the homogeneous boundary conditions of the Euler model based on duality of the differential operators involved. 
In this work we aim to go one step further in this direction by preserving the duality structure at the discrete level. 
This ensures a discrete projection property and allows to treat also inhomogeneous boundary conditions.

Our strategy ensures that the entropy of the whole scheme consisting of the prediction step and the correction step
is non-increasing. 
However, it does not prevent parasitic oscillations caused by boundary conditions not consistent with the Green--Naghdi equations. 
While consistence with well-posed boundary conditions remains an open problem, stability of the numerical scheme is the property we address here. 
Our approach provides a tool to numerically investigate  candidates for (well-posed) boundary conditions of the whole system of equations, to be confirmed by means of analysis. 
Apart from this, it gives rise to robust numerical methods to perform simulations for applications. 

To highlight the benefits of our strategy we present a simple full discretization. 
Here we focus on the discrete projection property rather than on robust computation of the pressure functions. 
The latter would require a discrete inf-sup condition to be satisfied independently of the time step size. 
Such a property is linked to the theory of stationary solutions and thus is deferred to future work. 
\medskip

Let us briefly present the structure of this article. 
In~\secref{sec:splitting} we introduce the Green--Naghdi equations. 
A time-discretization naturally leads to a splitting into an advection step including the shallow water equations, and a correction step. 
In~\secref{sec:proj-dt} we investigate the correction step for the time-discrete and space-continuous case. 
We formulate the correction step as projection for the whole space domain in~\secref{sec:proj-dt-unbd}. 
For a suitable choice of boundary conditions this property is preserved for bounded domains, cf.~\secref{sec:proj-dt-bd}.  
In~\secref{sec:fully-discr-bd} we investigate the fully discrete correction step.
We present a general strategy to construct a scheme with a discrete projection property for the whole space domain in~\secref{NumWholeSpace}. 
To demonstrate the benefits of this strategy we apply it for a simple collocated discretization. 
In~\secref{NumBounded} a condition on discrete boundary conditions is established that ensures that the scheme is still a projection for bounded domains. 
In~\secref{sec:full-num}, we propose a range of boundary conditions for the fully discrete scheme for the full Green--Naghdi equations that satisfy the previously established condition. 
Numerical evidence is presented to demonstrate the approach in some 1D situations, for which previously no strategy was available. 

\section{Prediction-correction splitting}\label{sec:splitting}

In this section we first present the system of Green--Naghdi equations in 1D and 2D and introduce the splitting based on a time discretization. 

\subsection{The Serre/Green--Naghdi model}\label{sec:GN-model}  
\begin{figure}\begin{center}
\begin{tikzpicture}
\begin{scope}[xscale=1,yscale=.3,xshift=-14cm]
	\newcommand\Bottom[1]{(.5-.4*cos(50*#1))+.2*sin(100*#1))}
	\newcommand\etaM[1]{(5+.1*(1+sin(200*#1)))}
	\newcommand\zetaM[2]{{(\Bottom{#2}+(#1)/\L*(\etaM{#2}-(\Bottom{#2})))}}
	\newcommand\qz[2]{((((#2-\etaM{#1}))/((\Bottom{#1}-\etaM{#1}))*((\qb+(((#2-\Bottom{#1}))*\qbb)))))}
\draw[->](-3.1,-.2)--(-1,-.2) node[below] {$x$};
\draw[->](-3.1,-.2)--(-3.1,{5/3-.2}) node[left] {$z$};
	\draw[thick,smooth,samples=100,domain=-3:6] plot(\x,{\etaM{\x}}) 
	%node[right]{\footnotesize$\eta$}
	;
	\fill[fill=black!30,smooth,samples=100,domain=-3:6] (-3,-.1)--plot(\x,{\Bottom{\x}})--(6,-.1);
	\draw[thick,smooth,samples=100,domain=-3:6] plot(\x,{\Bottom{\x}}) node[right]{\footnotesize$B$};
	\pgfmathsetmacro{\x}{3}\pgfmathsetmacro{\yyy}{\x-1.5}\pgfmathsetmacro{\y}{\etaM{\x}}\pgfmathsetmacro{\yy}{\Bottom{\x}}
	\pgfmathsetmacro{\qb}{1}\pgfmathsetmacro{\qbb}{.4}
	\draw[thick,<->] (\yyy,{\etaM{\yyy}})--(\yyy,{\Bottom{\yyy}}) node[pos=.6,right]{\footnotesize$h$};
	\draw[thick,->] (\x,{(\y+3*\yy)/4})--(\x+1.5,{(\y+3*\yy)/4}) node[pos=.75,above]{\footnotesize$\u$};
	\draw[thick,->] (\x,{(\y+3*\yy)/4})--(\x,{(\y+3*\yy)/4+2}) node[pos=.75,left]{\footnotesize$\w$};
	\draw[thick,>-<] (\x+.25,{(\y+3*\yy)/4+1.5})--(\x+.25,{(\y+3*\yy)/4+2.5}) node[midway,right]{\footnotesize$\sigma$};
	\draw[thick,dashed,domain=\etaM{\yyy}:\Bottom{\yyy}] plot({\yyy-\qz{\yyy}{\x})},\x) node[above right]{\footnotesize$q_B$} (.2,{.5*(\etaM{\yyy}+\Bottom{\yyy})})node{\footnotesize$\q$};
\end{scope}
\end{tikzpicture}
\caption{\label{fig:GN-unknown} 
 Illustration of unknowns in the Green--Naghdi model~\eqref{eq:GN} reproduced from~\cite{Parisot19}}
\end{center}\end{figure}
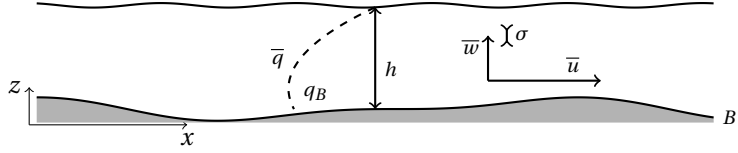

With temporal and spatial variables $(t,x) \in \R_+ \times\Omega$, where the space domain $\Omega$ is a subset of $\R^d$ in $d \in \left\{1,2\right\}$  dimensions, the system of Green--Naghdi equations~\cite{FPPS.2018} can be formulated as 
\begin{subequations}\label{eq:GN}
\begin{alignat}{3}\label{eq:GN-h}
\partial_th
&+
\diver \left(h\u\right)
&&=0, \\ \label{eq:GN-hu}
\partial_t \left(h\u\right)
&+
\diver \left(h\u\otimes\u+\frac g2h^2\I \right)\;
&&= - \nabla\left(h\q\right)
-(gh+q_B)\nabla B,
\\ \label{eq:GN-hw}
\partial_t (h\w )
&+
\diver \left(h\w\ \u \right)
&&=
q_B,
\\ \label{eq:GN-hs}
\partial_t (h\sigma)
&+
\diver \left(h\sigma\ \u\right)
&&=
\sqrt{3} \left(2\q-q_B\right). 
\end{alignat}
The given quantities are the gravitational constant $g$, the bathymetry $B(x) \in \mathbb{R}$ representing the bottom, and $\I$ denotes the identity matrix in $\mathbb{R}^{d \times d}$. 
The unknowns are the following: $h(t,x)\in\R_+$ is the water depth, $\u(t,x)\in\R^d$ is the vertically-averaged horizontal velocity, $\w(t,x) \in \R$ is the vertically-averaged vertical velocity, $\sigma(t,x)\in \R$ is the oriented vertical standard deviation of the vertical velocity.
The vertically-averaged hydrodynamic pressure $\q(t,x)\in \R$ and the hydrodynamic pressure at the bottom $q_B(t,x) \in \R$ are also unknowns of the Green--Naghdi model.
They can be interpreted as Lagrange multipliers to ensure that the following constraints are satisfied
\begin{align}\label{eq:GNconstraints}
\w = 
\u\cdot \nabla B-\frac{h}{2} \diver \u
\qquad \text{ and } \qquad 
\sigma =-\frac{h}{2\sqrt{3}}\diver \u. 
\end{align}
\end{subequations}
In \figref{fig:GN-unknown} an illustration of the unknown functions is presented.
Additionally, initial conditions $\left(h,\u\right)(0,x)=\left(h^0,\u^0\right)(x)$ have to be prescribed, while the initial vertical velocity $\left(\w,\sigma\right)(0,x)\eqqcolon \left(\w^0,\sigma^0\right)(x)$ are given as function of $\left(h^0,\u^0\right)(x)$ by the constraints~\eqref{eq:GNconstraints}.

Note that there are several formulations of the Green--Naghdi equations that are equivalent for sufficiently smooth solutions, cf.~\cite{CS.93, La.2013, Peregrine67, Z.1972}. 
The formulation in~\cite{Cienfuegos06} can be recovered from~\eqref{eq:GN} by inserting the constraints \eqref{eq:GNconstraints} into \eqref{eq:GN-hw} and \eqref{eq:GN-hs} and using the mass conservation \eqref{eq:GN-h} to find the hydrodynamic pressure functions expressed as functions of $h$, $\u$ and its derivatives.
We obtain $\q=h\left(\frac{\P}3+\frac{\Q}2\right)$ and $q_B=h\left(\frac{\P}2+\Q\right)$ with
\begin{align*}
\P=
-h\left(\partial_t\left(\diver \u\right)+\u\cdot\nabla\left( \diver \u\right)-\left(\diver \u\right)^2\right)
\quad\textrm{and}\quad
\Q=\nabla B\cdot\left(\partial_t\u+\u\cdot \nabla\u\right) + \nabla\left(\nabla B\right):\u\otimes\u.
\end{align*}
The formulation~\eqref{eq:GN} is a slightly weaker version in the sense that it contains more variables but only first order differential operators. 
It has the advantage that in the time-discrete form it exhibits a linear projection structure similarly as the incompressible Euler equations.

For the whole spatial domain $\Omega = \R^d$ some analysis on the Green--Naghdi model is performed in~\cite{La.2013}.
In~\cite{I.2011, Lannes08, L.2006} it is shown that the Green--Naghdi model is well-posed in a finite time cylinder $(0,T)\times \R^d$ such that at the maximal time $T$ either the water depth degenerates or the velocity tends to infinity. 
In addition, sufficiently regular solutions of~\eqref{eq:GN} satisfy the following energy conservation identity
\begin{align}\label{eq:GNenergy}
\partial_t\E\left(W\right)
+\nabla\cdot\left(\Gsw\left(W\right)+\Ggn\left(W,\q,q_B\right)\right)
=0,
\end{align}
where the state vector reads $W\coloneqq\left(h,h\u^\top,h\w,h\sigma,B\right)^\top$, the mechanical energy is 
\begin{align*}
\E\left(W\right)\coloneqq gh\left(B+\frac h2\right)+h\frac{\abs{u}^2+\w^2+\sigma^2}2,
\end{align*}
 and the fluxes are given as
\begin{align*}
\Gsw\left(W\right)\coloneqq \left(g\left(h+B\right)+\frac{\abs{\u}^2+\w^2+\sigma^2}2\right)h\u
\qquad\textrm{and}\qquad
\Ggn\left(W,\q,q_B\right) \coloneqq h\q\ \u.
\end{align*}
Here $\l|\bullet\r|$ denotes the Euclidean norm. 
Note that the bathymetry function $B$ in the state vector $W$ is given. 
Since the energy conservation plays a central role in this work we recall the proof in 1D for the reader's convenience.
Full details of the proof can be found in the literature, see for instance~\cite{La.2013}. For the proof of the precise form  \eqref{eq:GNenergy} see~\cite{FPPS.2018}.

\begin{proof}[Proof of energy conservation \eqref{eq:GNenergy} in 1D]
For simplicity of notation we present the proof in the case of one spatial dimension $d = 1$. 
Note that the case $d = 2$ can be treated analogously. 
Let us first introduce the vectorial form of \eqref{eq:GN} in 1D. 
It reads
\begin{equation}\label{eq:GN_vect}
%\partial_tW+\nabla\cdot F\left(W\right)=S\left(W\right)+Q\left(W,\q,q_B\right),
\partial_t W+A\left(W\right)\partial_x W=
Q\left(W,\q,q_B\right),
\end{equation}
with the matrix of the hyperbolic part and the dispersive term respectively
\begin{equation*}
A\left(W\right)\coloneqq\begin{pmatrix}
0&1&0&0&0
\\
gh -\u^2 &2 \u &0&0&gh
\\
-\u\,\w& \w&\u&0&0
\\
-\u\,\sigma& \sigma&0&\u&0
\\
0&0&0&0&0
\end{pmatrix}
\qquad\textrm{and}\qquad
Q\left(W,\q,q_B\right)\coloneqq 
\begin{pmatrix}
0
\\\ds
-\partial_x \left(h\q\right)-q_B\partial_x B
\\\ds
q_B
\\\ds
\sqrt{3}\left(2\q-q_B\right)
\\\ds
0
\end{pmatrix}.
\end{equation*}
Then \eqref{eq:GN_vect} supplemented by the constraints ~\eqref{eq:GNconstraints} is equivalent to \eqref{eq:GN} in 1D.  

Applying the definitions yields the compatibility identity
\begin{equation}\label{FluxCompatibility}
\nabla_W\E\ A=\nabla_W\Gsw,
\end{equation}
where $\nabla_W\bullet$ denotes the gradient with respect to the state variable $W$. %
Taking the scalar product with $\nabla_W\E$ in \eqref{eq:GN_vect} shows that
\begin{align*}
\partial_t\E+\partial_x\Gsw 
&= \nabla_W\E\cdot\left(\partial_tW+A\partial_x W\right) 
= \nabla_W\E\cdot Q\\
&= -\partial_x\Ggn
+\q\left(h\partial_x\u+2\sqrt{3}\sigma\right)
+q_B\left(\u\partial_x B+\w-\sqrt{3}\sigma\right).
%+\q h\nabla\cdot\u
%-q_B\u\cdot\nabla B
%+q_B\w
%+\sqrt{3}\left(2\q-q_B\right)\sigma
\end{align*}
Then, we conclude using the constraints \eqref{eq:GNconstraints}. 
\end{proof}

For hyperbolic equations compatibility relations of the form  \eqref{FluxCompatibility} are well-known, cf.~\cite[\S1.4]{Bouchut04}. 
In that context $(\E,\Gsw)$ is referred to as entropy / entropy flux pair. 
The existence of such a pair is in general required for uniqueness of weak solutions, even though it is not always sufficient. 
Numerical schemes that preserve the entropy stability are referred to as \emph{entropy-satisfying} and are very robust, see \cite{Bouchut04}. 

Similarly as in the case of the shallow water model, for the Green--Naghdi equations the mechanical energy $\E$ satisfies the properties of a mathematical entropy, i.e., it is a convex function in the first four arguments of $W$ (without the given component $B$) and it satisfies a conservation law \eqref{eq:GNenergy}.
Here we aim for numerical schemes on bounded domains that are entropy-satisfying with the mechanical energy acting as a mathematical entropy.
The proof of the energy conservation \eqref{eq:GNenergy} is based on the existence of a flux $\Ggn$ such that $\nabla_W\E\cdot Q=-\nabla\cdot \Ggn$. 
At first sight this property seems to be non-trivial, but it is strongly linked to the duality between the dispersive terms and the constraints \eqref{eq:GNconstraints} as we shall see in \secref{sec:proj-dt}.
This structure corresponds to the one of incompressible fluid equations, such as the incompressible Euler equations.

\subsection{The time-discrete problem}\label{sec:time-discrete}  
In the following we consider a time-discrete version of the Green--Naghdi equations without discretization in the spatial domain. 
For the time stepping we set $t^0 = 0$ and $t^{n+1} = t^n + \dtn$ with time step $\dtn>0$.
The choice of $\dtn$ is further discussed in~\secref{sec:full-num}.

The time-discrete problem can be decomposed into two steps per time step.
This results in a prediction-correction type approach similar to the one for the incompressible Euler equations, cf.~\cite{Guermond06}.
Each time step can be formulated as composition of 
an explicit shallow water and advection step,
and an implicit correction step ensuring the constraints by means of the pressure functions $\q$ and $q_B$.
Such a splitting into an advection step and a correction step is used in a number of contributions in particular for numerical computations, see~\cite{SainteMarie16,BCLMT.2011,Gavrilyuk17,Parisot19,Popinet20}.

\begin{enumerate}[label=(\Roman*),leftmargin=0cm]\setlength{\itemindent}{1cm}
\item {\style Advection step:}
For given $\left(h^n,\u^n,\w^n,\sigma^n\right)$, let $\left(h^{n*},\u^{n*},\w^{n*},\sigma^{n*}\right)$ be such that
%For given $W^n=\left(h^n,h^n\u^n,h^n\w^n,h^n\sigma^n,B\right)$,\\let $W^{n*}=\left(h^{n*},h^{n*}\u^{n*},h^{n*}\w^{n*},h^{n*}\sigma^{n*},B\right)$ be such that
%
\begin{subequations}\label{eq:GN-dt-adv}
\begin{alignat}{3}
h^{n*} & = h^n 
&&- \dtn \diver
 \(h^n\u^n\)
 %F_{h}^n
, \label{eq:GN-dt-adv-h}
&&\\
h^{n*} \u^{n*}
& = h^n  \u^n 
 && - \dtn  \diver 
\(h^n\u^n\otimes\u^n+\frac g2\l|h^n\r|^2\I\)
%F_{h \u}^n
%&& + \dtn S^n, \label{eq:GN-dt-adv-hu} \\
&& - \dtn gh^n\nabla B, \label{eq:GN-dt-adv-hu} \\
h^{n*}\w^{n*}
&= 
h^n \w^n 
&& - \dtn \diver
\(h^n\w^n\ \u^n\)
%F^n_{h\w}
, \label{eq:GN-dt-adv-w}\\
h^{n*} \sigma^{n*}
&= 
h^n \sigma^n 
&& - \dtn \diver
\(h^n\sigma^n\ \u^n\)
%F^n_{h\sigma}
. \label{eq:GN-dt-adv-s}
\end{alignat}
\end{subequations}
Hence, the functions $(h^{n*}, h^{n*} \u^{n*})$ are given as solutions to the explicit time-discrete shallow water system with bathymetry source term consisting of \eqref{eq:GN-dt-adv-h} and \eqref{eq:GN-dt-adv-hu}.
Then $(h^{n*} \w^{n*},  h^{n*} \sigma^{n*})$ are solutions to the system of advection equations composed by \eqref{eq:GN-dt-adv-w} and \eqref{eq:GN-dt-adv-s}, which is of the same form as the one describing the transport of a passive pollutant. 

Denoting $W^n \coloneqq \left(h^n,h^n\u^n,h^n\w^n,h^n\sigma^n,B\right)$ and $W^{n*}$ analogously, 
similar computations and smoothness assumptions as the ones in the proof of \eqref{eq:GNenergy} lead to the following balance law
\begin{equation}\label{Entropy_adv}
\E\left(W^{n*}\right)-\E\left(W^n\right)+\delta_t^n\nabla\cdot\Gsw\left(W^n\right)=\delta_t^n\Rest_{\mathrm{adv}}^n.
\end{equation}
Here $\Rest_{\mathrm{adv}}^n$ denotes the time discretization error that vanishes as $\delta_t^n$ tends to zero. 

The hyperbolic system \eqref{eq:GN-dt-adv} can be numerically approximated by standard methods, cf.~\cite{Bouchut04}.
Some of them, for example the HLL solver with an improved version of the hydrostatic reconstruction \cite{Berthon19}, are entropy-satisfying, i.e., the approximate solutions satisfy a fully discrete version of \eqref{Entropy_adv} with $\Rest_{\mathrm{adv}}^n\le0$.

\item {\style Correction step:}
Then let $h^{n+1} \coloneqq h^{n*}$ and let the functions $\left(\u^{n+1},\w^{n+1},\sigma^{n+1}\right)$ and $\left(\q^{n+1}, q_B^{n+1}\right)$ be determined by 
\begin{subequations}\label{eq:GN-dt-proj-pre}
\begin{alignat}{3}
h^{n+1} \u^{n+1} 
& = h^{n+1} \u^{n*}
&& - \dtn \( \nabla \(h^{n+1} \q^{n+1}\) + q_B^{n+1} \nabla B \), \label{eq:GN-dt-proj-u-pre} \\
h^{n+1} \w^{n+1} 
& = h^{n+1} \w^{n*}
&& + \dtn q_B^{n+1}, \label{eq:GN-dt-proj-w-pre}\\
h^{n+1} \sigma^{n+1} 
& = h^{n+1} \sigma^{n*}
&& +\dtn \sqrt{3}\(2\q^{n+1} - q_B^{n+1}\),\label{eq:GN-dt-proj-s-pre}
\end{alignat} 
subject to the constraints
\begin{align}\label{eq:GNconstraints-dt}
\w^{n+1} = \u^{n+1}\cdot \nabla B -\frac{h^{n+1}}{2} \diver \u^{n+1} \qquad \text{ and } \qquad 
\sigma^{n+1} =-\frac{h^{n+1}}{2\sqrt{3}}\diver \u^{n+1}. 
\end{align}
\end{subequations}

Similar computations and smoothness assumptions as the ones in the proof of \eqref{eq:GNenergy} lead to the following balance law 
\begin{equation}\label{Entropy_proj}
%\mat{c}{\ds
%\E\left(h^{n+1},h^{n+1}\u^{n+1},h^{n+1}\w^{n+1},h^{n+1}\sigma^{n+1}\right)-\E\left(h^{n*},h^{n*}\u^{n*},h^{n*}\w^{n*},h^{n*}\sigma^{n*}\right)
%~\qquad\\\qquad~\ds
%+\delta_t^n\nabla\cdot\left(h^{n+1}\q^{n+1}\u^{n+1}\right)
%=\delta_t^n\Rest_{\mathrm{cor}}^{n+1},
%}
\E\left(W^{n+1}\right)-\E\left(W^{n*}\right)
+\delta_t^n\nabla\cdot \Ggn\left(W^{n+1},\q^{n+1},q_B^{n+1}\right)
=\delta_t^n\Rest_{\mathrm{cor}}^{n+1},
\end{equation}
and the time discretization error $\Rest_{\mathrm{cor}}^{n+1}\le0$ vanishes as $\delta_t^n$ tends to zero.   

In the following we aim to preserve this energy balance both for weaker notions of solutions in integrated form as well as for the fully discrete case. 
In conjunction with the corresponding estimate for the hyperbolic step this will ensure entropy stability of the full scheme. 
\end{enumerate}

\begin{rmk}\label{rmk:h-fixed}
The water depth $h^{n+1}$ is fully determined by the advection step. 
In the correction step it merely serves as a parameter, and hence the constraints are linear in the unknown functions $\left(\u^{n+1}, \w^{n+1}, \sigma^{n+1}\right)$. 
This is a major benefit of the splitting strategy.
%
%In addition, the potential energy $gh\(B+\frac h2\)$ is hence exactly preserved by the correction step.
%The balance law \eqref{Entropy_proj} 
%
\end{rmk}
 
As already observed in~\cite{SainteMarie16, Parisot19, Popinet20}, the system~\eqref{eq:GN-dt-proj-pre} is related to a projection of the solutions of the advection step~\eqref{eq:GN-dt-adv} to the set of admissible solutions satisfying the constraints~\eqref{eq:GNconstraints-dt}.
In fact, the pressure functions $\left(\q^{n+1},q_B^{n+1}\right)$ act as Lagrange multipliers enforcing the constraints. 
In the following section we investigate the correction step and its projection structure in more detail.

\section{Correction step of the time-discrete problem}
\label{sec:proj-dt}

In this section we analyze a single correction step. 
We shall find that indeed it has projection structure for the whole space domain $\Omega = \R^d$, see~\secref{sec:proj-dt-unbd}.
On bounded sets $\Omega \subset \R^d$ we characterize boundary conditions for which a projection structure is available, see~\secref{sec:proj-dt-bd}. 

Let us recall the system~\eqref{eq:GN-dt-proj-pre} and for simplicity we avoid the time step indices.  
Within the correction step the water depth function is fixed and can be seen as parameter $h(x)$ similarly as the bathymetry $B(x)$, see \rmkref{rmk:h-fixed}.
Also the time step $\dt>0$ is fixed already by the preceding advection step.
We want to find functions $U=\left(\u,\w,\sigma\right)^\top$ and $\left(\q,q_B\right)$ such that for the given function $U^*=\left(\u^*,\w^*,\sigma^*\right)^\top$ we have that
\begin{subequations}\label{eq:GN-dt-proj}
\begin{equation}\label{eq:GN-dt-proj-U}
U=U^*-\dt\Psi_h\left(\q,q_B\right)
\qquad\text{ with }\qquad
\Psi_h\left(\q,q_B\right)
\coloneqq
\frac{1}{h}
\begin{pmatrix}
\nabla\left(h\q\right) +q_B \nabla B
\\
- q_B
\\
- \sqrt{3}(2 \q - q_B)
\end{pmatrix},
\end{equation}
subject to the constraints 
\begin{align}\label{eq:GNconstraints-dt-proj}
\w = \u\cdot \nabla B -\frac{h}{2} \diver \u, \quad \text{ and } \quad 
\sigma =-\frac{h}{2\sqrt{3}}\diver \u.
\end{align}
\end{subequations}

To investigate the projection structure, let us introduce the formal framework including some assumptions on the fixed functions. 
Let the spaces $L^p(\Omega)$ and $W^{1,p}(\Omega)$ be the standard Lebesgue and Sobolev spaces, for $p \in [1,\infty]$ and an open set $\Omega \subset \R^d$.  
For $p=\infty$ those are the spaces of essentially bounded functions, and of Lipschitz functions, respectively. 

\begin{hyp}[Parameters]\label{hyp:B-h-dt}
Assume that:
\begin{enumerate}[label=\roman*),ref=\roman*)]
\item \label{hyprefH}\global\def\hyprefH{\hypref{hyp:B-h-dt}.\ref{hyprefH}\xspace} $h \geq 0$,
$h \in L^{\infty}(\Omega)$ and~~$1/h\in L^{\infty}(\Omega)$;
\item \label{hyprefB}\global\def\hyprefB{\hypref{hyp:B-h-dt}.\ref{hyprefB}\xspace}
$B \in W^{1,\infty}(\Omega)$;
%\\which means that $B$ is a Lipschitz continuous space function;
%
\item \label{hyprefD}\global\def\hyprefD{\hypref{hyp:B-h-dt}.\ref{hyprefD}\xspace}
$\dt>0$.
\end{enumerate}
\end{hyp}

 \hyprefH means that $h$ is a positive function on $\Omega$, it is essentially bounded and bounded away from zero.  
This is in general not the case for solutions of the advection step~\eqref{eq:GN-dt-adv}, since so-called dry areas may occur. 
Those are subsets of $\Omega$ on which $h= 0$ and they are of great importance in practice.  
In their presence the correction step would have to be restricted to a subset compactly contained in the support of $h$. 
To reduce the level of technicality in this section we focus on non-degenerating $h$ as in \hyprefH.
We shall see in~\secref{sec:fully-discr-bd} that for the fully discrete problem it is not an issue to deal with dry areas. 

\subsection{\texorpdfstring{Whole space domain $\Omega=\R^d$}{Whole space domain}}

\label{sec:proj-dt-unbd}

In this section we consider the system of equations~\eqref{eq:GN-dt-proj} on the whole space $\Omega=\R^d$ to prepare and motivate the subsequent approach for bounded domains, see~\secref{sec:proj-dt-bd}.  

Thanks to \hyprefH the water depth $h$ is a weight function on $\Omega$ as defined in measure theory. 
Hence, we may work with the following weighted scalar product
\begin{equation*}
\ll f,g\rr_{h} \coloneqq \int_{\Omega} h f g \dx,
\end{equation*}
the induced norm $\norm{\cdot}_{h}$, and $\Lh$ the space of measurable function on $\Omega$ with bounded $\norm{\cdot}_{h}$-norm. 
The weighted norm $\norm{\cdot}_{h}$ appears naturally in the energy estimates~\eqref{eq:GNenergy} and also in the estimate~\eqref{Entropy_adv} for the advection equations. 
Hence, it is natural to assume that  
\begin{equation*}
U^* = \left(\u^*,\w^*,\sigma^*\right)^\top \in \Lh^d \times \Lh \times \Lh\eqqcolon \Lh^{d+2}.
\end{equation*}

The constraints in~\eqref{eq:GNconstraints-dt-proj} define a space of admissible functions, which any solution $U$ of~\eqref{eq:GN-dt-proj} is contained in.  
\begin{defi}[Space of admissible functions]
\label{def:Ah}
For any open set $\Omega\subset\R^d$, we define the \textit{space of admissible functions} on $\Omega$ by 
\begin{equation*}
\Ahb \coloneqq \left\{
U=(u,w,\sigma)^\top\in \Lh^{d+2} \colon 
\quad w = u \cdot \nabla B - \frac{h}{2} \diver u,
\quad\sigma = - \frac{h}{2 \sqrt{3}} \diver u \;
 \right\}.
\end{equation*}
\end{defi}
Furthermore, for any closed linear subspace~~$\Eset \subset\Lh^{d+2}$
we denote by $\Pi_h[\Eset]$  the linear $\ll \cdot, \cdot \rr_{h}$-orthogonal projection mapping $\Lh^{d+2}$ to $\Eset$, defined by
\begin{equation}\label{def:Pi-2}
\ll \Pi_h[\Eset](U),V \rr_{h} = \ll U,V \rr_{h} 
\qquad \text{ for all }  \; V \in\Eset. 
\end{equation}
The space $\Lh^{d+2}$ is a Hilbert space and by  \hypref{hyp:B-h-dt} one can show that $\Ahb$ is a closed linear subspace of $\Lh^{d+2}$. 
Consequently, the projection $\Pih$ is well-defined. 
By $\Ahb^\perp$ we denote the $\ll\cdot,\cdot\rr_h$-orthogonal complement of $\Ahb$ in $\Lh^{d+2}$. 

Now we are in the position to present the projection structure of~\eqref{eq:GN-dt-proj}. 
\begin{lem}[Projection property on $\Omega=\R^d$]\label{lem:proj-Rn}
Let \hypref{hyp:B-h-dt} be satisfied on $\Omega = \R^d$.  
Let a function $U^*\in \LRh^{d+2}$ be given. 
Assume that $U \in \LRh^{d+2}$ and functions $\left(\q,q_B\right) \in \LRh$ such that $\nabla\left(h \q\right) \in L^2(\R^d)^{d}$ form a solution to the correction step~\eqref{eq:GN-dt-proj}.  
Then, the solution is determined by the projection to the space of admissible functions in the sense that
\begin{equation*}
U = \Pih\left(U^*\right)
\qquad\textrm{and}\qquad
%%\Img(\Psi_h)=\Ahb^\perp.
\Psi_h\left(\q,q_B\right) \in \Ahb^\perp.
\end{equation*}
\end{lem}
\begin{proof}
By the constraints~\eqref{eq:GNconstraints-dt-proj}
we have that $U \in \Ahb$. 
Furthermore, if~\eqref{eq:GN-dt-proj} is satisfied, then we have for all $V =(V_1,V_2,V_3)^\top \in \Ahb$ that 
\begin{equation} \begin{split} \label{eq:proj-Rn}
\ll U - U^*, V \rr_{h} 
&=
- \dt
\int_{\R^d}  \diver \left(h \q  V_1\right)  \dx
= 0. 
\end{split}
\end{equation}
Note that by assumption we have that $\nabla\left(h\q\right) \in L^2(\R^d)^d$ and since $V\in \Ahb$ using \hyprefH we find that $\diver V_1 \in L^2(\R^d)$. 
The fact that the integral vanishes follows by smooth approximation, the Gau\ss{}--Green theorem, and decay properties of integrable functions on $\R^d$. 
This proves that $\Psi_h\left(\q,q_B\right) \in \Ahb^\perp$ and by uniqueness of the decomposition it follows that $U = \Pih\left(U^*\right)$. 
\end{proof}

The projection structure has the following benefits: 
The space-continuous problem is well-posed and using the orthogonality the following energy balance holds
\begin{equation}\label{eq:energy}
\norm{U}_h^2 - \norm{U^*}_h^2 = - \dt^2 \norm{\Psi_h\left(\q,q_B\right)}_h^2.
\end{equation}
Since the water depth is not modified by the correction step, this equality is a space integrated version of the balance law \eqref{Entropy_proj} with $\Rest_{\mathrm{cor}}^{n+1}=-\dt\norm{\Psi_h\left(\q,q_B\right)}_h^2$. 
For numerical computations a space-discrete projection property is particularly useful, since it guarantees numerical stability. 
More precisely, a discrete version of the identity~\eqref{eq:energy} ensures that the scheme for the space discrete correction step is entropy-satisfying.
In addition, the projection property paves the way to efficient higher-order schemes requiring only one implicit correction step, cf.~\cite{Guermond06}. 
This strategy has been applied to the Green--Naghdi model in~\cite{Parisot19} for symmetric boundary conditions and numerical evidence shows that the second order is achieved. 
Last but not least, in the following section we adopt the projection point of view to identify a family of boundary conditions for which well-posedness of the correction step is guaranteed.
\medskip

Functions in $\Ahb$ and $\Ahb^\perp$ enjoy regularity properties in the sense that the hypothesis of \lemref{lem:proj-Rn} are satisfied without extra assumptions.  
In order to show this let us consider the classical function spaces for an open subset $\Omega \subset \R^d$ defined by
\begin{align*}
H^1(\Omega)&\coloneqq\left\{f \in L^2(\Omega) \colon \nabla f \in L^2(\Omega)^d\right\},
\\\ds
\Hdiv &\coloneqq \left\{f \in L^2(\Omega)^d \colon  \diver f \in L^2(\Omega)\right\},
\end{align*}
and recall that by \hyprefH the function spaces $L^2(\Omega)$ and $\Lh$ coincide. 
\begin{prop}\label{prop:reg-Ah}
Let \hypref{hyp:B-h-dt} be satisfied for an open subset $\Omega \subset \R^d$. 
\begin{enumerate}[label=\roman*),ref=\roman*)]
\item \label{itm:reg-Ah-U}
For any $U = \left(\u,\w,\sigma\right)^\top \in \Ahb$ we have that $\u\in \Hdiv$.
\item \label{itm:reg-Ah-q}
For any~~$\Phi \in \Ahb^\perp$ there exists a  unique pair of functions $\left(\q,q_B\right)\in L^2(\Omega)^2$ such that $\Psi_h\left(\q,q_B\right)=\Phi$ with $\Psi_h$ as in~\eqref{eq:GN-dt-proj-U}. In addition, one has that $h\q\in H^1(\Omega)$.
\item \label{itm:reg-Ah-qinv}
Conversely, for any $\left(\q,q_B\right)\in L^2(\Omega)^2$ with $h\q \in H^1(\Omega)$ one has $\Psi_h\left(\q,q_B\right)\in \Ahb^\perp$. 
\end{enumerate}
\end{prop}
\begin{proof}
The proof of \ref{itm:reg-Ah-U} follows from the second constraint using that $\sigma \in\Lh$ and $1/h \in L^{\infty}(\Omega)$. 
Further, \ref{itm:reg-Ah-qinv} follows from the proof of \lemref{lem:proj-Rn}. 
It remains to prove \ref{itm:reg-Ah-q}.

For a function $\Phi  = (\phi_1, \phi_2,\phi_3)^\top \in \Ahb^\perp \subset \Lh^{d+2}$ we choose $\left(\q,q_B\right)\in L^2(\Omega)$ as
\begin{align*}
q_B \coloneqq - h \phi_2
\qquad\textrm{and}\qquad
\q \coloneqq  - \frac{h}{2} \left(\phi_2 + \frac{\phi_3}{\sqrt{3}}\right). 
\end{align*}
By a direct computation we find that $h \phi_3 = - \sqrt{3}\left(2 \q - q_B\right)$ which agrees with the third component of $\Psi_h\left(\q,q_B\right)$ and uniqueness is given. 
Hence, it remains to identify the first components. 
Since $\Phi\in\Ahb^\perp$, we have for any $V = (V_1,V_2,V_3)^\top \in\Ahb$ that
\begin{equation}\label{eq:reg-Ah}
\mat{r@{~}c@{~}l}{\ds
0 
&=&\ds
\ll V,\Phi \rr_{h}
%\\
%&=&\ds
%\ll \left( V_1,V_1 \cdot \nabla B -\frac{h}{2} \nabla \cdot V_1, -\frac{h}{2\sqrt{3}} \nabla \cdot V_1\right)^\top,\Phi \rr_{h}
%\\&=&\ds
=
\int_{\Omega}  V_1 \cdot h\left(\phi_1   +  \phi_2  \nabla B\right) \dx   
- \int_{\Omega}  \frac{h^2}2 \left(\phi_2 + \frac{1}{\sqrt{3}}  \phi_3\right)  \nabla \cdot V_1  \dx.
}\end{equation}
This implies that $h(\phi_1   +   \phi_2   \nabla B) \in L^2(\Omega)^d$ is the weak gradient of 
$-\frac{h^2}2\left(\phi_2 + \frac{1}{\sqrt{3}}  \phi_3\right) = h \q$. 
It follows that $h\q \in H^1(\Omega)$ and that
\begin{equation*}
h\phi_1
=-\nabla\(\frac{h^2}2\(\phi_2 + \frac{1}{\sqrt{3}}  \phi_3\)\)-h\phi_2   \nabla B
=\nabla\(h\q\)+q_B\nabla B.
\end{equation*}
This identifies the first component of $\Psi_h\left(\q,q_B\right)$ and finishes the proof.
\end{proof}
Thanks to \propref{prop:reg-Ah} the inverse mapping $\Psi_h^{-1}\colon \Ahb^\perp \to  L^2(\Omega)^2$ with $\Psi_h^{-1}\circ \Psi_h=\I$ exists. 
In fact, it is given by 
\begin{equation}\label{def:psi-inv}
\Psi_h^{-1} \left( \Phi\right) 
=-h\left(
\frac{1}{2} \left(\phi_2 + \frac{\phi_3}{\sqrt{3}}\right), \phi_2 \right) \qquad \text{ for } \Phi = \left(\phi_1, \phi_2,\phi_3\right)^\top \in \Ahb^\perp. 
\end{equation}
By linearity of the system of equations~\eqref{eq:GN-dt-proj} and orthogonality, uniqueness of solutions follows.  
Thus, by \lemref{lem:proj-Rn} and \lemref{prop:reg-Ah} we have that the unique solution is given by the projection and no extra assumption on the Sobolev regularity is needed.  

\begin{lem}[Well-posedness on $\Omega = \R^d$]\label{lem:wellp-Rd} 
Let \hypref{hyp:B-h-dt} be satisfied on $\Omega = \R^d$. 

 Then, for any $U^* \in \LRh^{d+2}$ there exists a unique solution to the correction step~\eqref{eq:GN-dt-proj} on $\Omega$, consisting of functions~~$U \in \Ahb$ and $\left(\q,q_B\right) \in L^2(\R)^2$ with $h\q \in H^1(\R^d)$. 
The solution is given by 
\begin{equation*}
U = \Pih(U^*) \qquad \text{ and } \qquad \left( \q, q_B \right) = \Psi_h^{-1}\left(\frac{U^*-U}{\dt}\right).  
\end{equation*}
\end{lem}
In the next section we aim for a similar result for bounded spatial domains. 

\begin{rmk}[Weighted spaces]
It is possible to make sense of weighted spaces $L^2(\Omega;\omega)$ for weaker notions of weight functions $\omega$ than the one we assume for $h$ in \hyprefH, cf.~\cite{HKM.2006}. 
In fact, by \hyprefH the weighted space $\Lh$ and the standard Lebesgue space $L^2(\Omega)$ agree. 
Without using \hyprefH for $U \in \Ahb$ and $\Psi_h\left(\q,q_B\right) \in  \Ahb^\perp$ as above one finds that
\begin{equation*}
\nabla\cdot\u\in L^2\left(\Omega;h^3\right),
\quad
q_B\in L^2\left(\Omega;1/h\right),
\quad 
h\q \in L^2\left(\Omega;1/h^3\right)
\quad\textrm{and}\quad 
\nabla\left(h\q\right)\in L^2\left(\Omega;1/h\right)^d. 
\end{equation*} 
Recall that under \hyprefH this is equivalent to the statement in \propref{prop:reg-Ah}. 
\end{rmk}
 
We choose the setting of unweighted Sobolev spaces in order to have classical trace theory at hand in~\secref{sec:proj-dt-bd}. 

\subsection{\texorpdfstring{Bounded spatial domain $\Omega\subset\R^d$}{Bounded spatial domain}}

\label{sec:proj-dt-bd}

In this section we identify a class of boundary conditions for which the projection structure of the system of equations encountered in~\secref{sec:proj-dt-unbd} is maintained. 
More specifically, we aim to pose boundary conditions yielding well-posedness analogously as in \lemref{lem:wellp-Rd}. 
We consider an open bounded set $\Omega \subset \R^d$ with Lipschitz boundary $\partial \Omega$. 
Recall that the proof of the projection structure in \lemref{lem:proj-Rn} relies on the Gau\ss--Green Theorem with boundary terms vanishing at infinity. 
In the case of a bounded domain the remaining term in the duality relation is 
\begin{equation}\label{eq:bc-motiv-boundaryterms}
-\int_{\Omega} \diver\left( h \q \; V_1\right) \dx = -\int_{\partial\Omega} h \q \, V_1 \cdot \nu \,\d s(x),
\end{equation}
for any $V = (V_1,V_2,V_3)^\top \in \Ahb $ and any $\Phi =\left(\Phi_1,\Phi_2,\Phi_3\right) \in \textrm{Im}\left(\Psi_h\right)$ with $\q$ the first component of $\Psi_h^{-1}\left(\Phi\right)$, provided the traces are sufficiently smooth.
Here $\nu$ denotes the outer unit normal on $\partial \Omega$. 
This motivates our choice of boundary conditions preserving a projection structure for homogeneous boundary conditions, which is available in the case of the whole space domain. 
For this we supplement the definition of the admissible space $\Ahb$ and the pressure functions by certain zero boundary values such that
\begin{equation*}%\label{eq:bc-motiv-duality}
\ll V,\Phi \rr_{h} = 0,
\end{equation*}
for all $V\in \Ahb$ and any $\Phi \in \textrm{Im}\left(\Psi_h\right)$ that satisfy those homogeneous boundary conditions to be specified in the following.

For the scope of this section, we assume that the boundary $\partial\Omega$ can be decomposed into subsets on each of which one of the factors of the integrand vanishes. 
\begin{hyp}[Decomposition of the boundary $\partial \Omega$]\label{hyp:dom-dt}
Let~~$\Omega \subset \mathbb{R}^d$ be an open bounded and connected set with  Lipschitz boundary $\partial \Omega$. 
Assume that there exist relatively open sets $\Gamma_u, \Gamma_{hq} \subset \partial \Omega$ with finitely many connected components decomposing the boundary $\partial\Omega$, i.e., we have that~~$\partial\Omega =\overline{\Gamma_u}\cup\overline{\Gamma_{hq}} $ and~~$\Gamma_{u}\cap\Gamma_{hq}=\emptyset$. 
\end{hyp}
A domain $\Omega$ with finitely many connected components can be treated componentwise, and hence the assumption of $\Omega$ being connected is not very restrictive. 
Note that the decomposition of $\partial \Omega$ is considered as given for the correction step but is fixed within one time step in the same sense as $h$, cf.~\rmkref{rmk:h-fixed}. 
In~\secref{sec:fully-discr-bd} we discuss possible choices of decomposition for specific applications.  

Now let us consider the following formal boundary conditions
\begin{equation}\label{eq:bc-dt-inhom}
\u \cdot \nu|_{\Gamma_u} = \bcu
\qquad\textrm{and}\qquad
h\q|_{\Gamma_{hq}} = \bchq
 \end{equation}
for given functions $\bcu\colon\Gamma_u\to\R$ and $\bchq\colon\Gamma_{hq}\to\R$. 

Let us comment on the function space framework in which we formulate the boundary conditions. 
Any function in $ H^1(\Omega)$ admits a trace in $H^{1/2}(\partial \Omega)$, and the trace operator mapping $ H^1(\Omega)$ to $H^{1/2}(\partial \Omega)$ is linear, bounded and onto. 
Similarly, the normal trace space of $\Hdiv$ is $H^{-1/2}(\partial\Omega)$. 
Here $H^{-1/2}(\partial \Omega)$ is the dual space of $H^{1/2}(\partial \Omega)$ with respect to $L^2(\partial \Omega)$ and we denote the duality relation by $\ll \cdot, \cdot \rr_{H^{-1/2}(\partial \Omega), H^{1/2}(\partial \Omega)}$. 
Also the trace operator mapping $ \Hdiv$ to $H^{-1/2}(\partial\Omega)$ is linear, bounded and onto. 
The following integration by parts formula holds for any $v \in \Hdiv$ and $\phi \in H^{1}(\Omega)$
\begin{align}\label{eq:ibp-dual}
\int_{\Omega} v \cdot \nabla \phi \dx + \int_{\Omega} \diver v \, \phi \dx = 
\ll v \cdot \nu, \phi \rr_{H^{-1/2}(\partial \Omega), H^{1/2}(\partial \Omega)}.  
\end{align}
cf.~\cite[Ch.~I.2.2]{GR.1986}, where for readability we do not introduce notation for the traces operators. 
Note that $H^{1/2}(\partial \Omega) \subset L^2(\partial \Omega)$, and hence trace functions can be restricted to $\Gamma_{hq} \subset \partial \Omega$.  
Since $h\q \in H^1(\Omega)$ we thus may consider boundary data $\bchq \in H^{1/2}(\Gamma_{hq})$. 
However, distributions in $H^{-1/2}(\partial \Omega)$ cannot in general be restricted to subsets of $\partial \Omega$ since the respective trace operator is not continuous, cf.~\cite{GH.2015.S}. 
Thanks to the properties of functions in the admissible set $\Ahb$ given by \lemref{prop:reg-Ah} we have that $\u\in \Hdiv$ and thus it has a trace in $H^{-1/2}\left(\partial\Omega\right)$. 
To impose data on $\u \cdot \nu |_{\Gamma_u}$ in a suitable sense we have to specify how to impose conditions on a restriction of a distribution in $H^{-1/2}\left(\partial\Omega\right)$ to a part of the boundary in a suitably weak sense.
\medskip

As indicated by~\eqref{eq:bc-motiv-boundaryterms} and~\eqref{eq:ibp-dual} homogeneous boundary conditions, i.e., choosing $\bcu \coloneqq 0$ and $\bchq \coloneqq 0$ ensure the projection property since then the boundary term vanishes, see~\secref{sec:dt-bd-hom} below. 
The case of inhomogeneous boundary conditions can be reduced to the one of homogeneous boundary conditions by the standard approach of reference functions, see~\secref{sec:dt-bd-inhom}. 

\subsubsection{Homogeneous boundary conditions}\label{sec:dt-bd-hom}
We consider homogeneous boundary conditions in~\eqref{eq:bc-dt-inhom}, i.e.,
\begin{align}\label{eq:bc-dt-hom}
\u \cdot \nu|_{\Gamma_u} = 0
\qquad\textrm{and}\qquad
h\q|_{\Gamma_{hq}} = 0,
\end{align}
in a sense of traces to be specified.

Let us first introduce the spaces with homogeneous conditions on the respective traces. 
For this purpose let $\Gamma \subset \partial \Omega$ be a relatively open subset with finitely many connected components and let $\Gc=(\partial \Omega \setminus \Gamma)^\circ$ be the relative open complement of $\Gamma$ in $\partial \Omega$. 
Now we define the subspaces $H^1_{\Gc}(\Omega) \subset  H^1(\Omega)$ and $H_{\Gamma}(\divergence;\Omega) \subset \Hdiv$, by
\begin{align*}
H^1_{\Gc}(\Omega) 
&\coloneqq 
\left\{f \in H^1(\Omega)\colon f|_{\Gc} = 0 \;\;\text{ as trace in } H^{1/2}\left(\Gc\right)\right\},
\\
H_{\Gamma}(\divergence;\Omega)
&\coloneqq 
\left\{v \in \Hdiv \colon \ll v \cdot \nu, f \rr_{H^{-1/2}(\partial \Omega), H^{1/2}(\partial \Omega)} = 0 \;\; \text{ for all } f \in H^1_{\Gc}(\Omega)\right\}. 
\end{align*}
Note that by the continuity of the trace operators both spaces $H^1_{\Gc}(\Omega)$ and $H_{\Gamma}(\divergence;\Omega)$ are closed. 
Now we may introduce the space of admissible functions with a homogeneous condition on the normal trace. 
\begin{defi}[Space of admissible functions with homogeneous trace condition]\label{def:Ah-Pi-bd-dt-hom} 
For an open bounded set~~$\Omega \subset \R^d$ with Lipschitz boundary $\partial \Omega$ let $\Ahb$ be as in \defref{def:Ah}. 
For a relatively open subset $\Gamma \subset \partial \Omega$ with finitely many connected components
we denote the space of admissible functions with homogeneous condition on the normal trace by
\begin{equation*}
\Ahg \coloneqq \Ahb \cap \left(H_{\Gamma}(\divergence;\Omega) \times L^2(\Omega;h) \times L^2(\Omega;h)\right).
\end{equation*} 
\end{defi}
Due to the fact that both $\Ahb$ and $H_{\Gamma}(\divergence;\Omega)$ are closed subspaces of the respective spaces it follows that $\Ahg$ is a closed linear subspace of $\Ahb\subset\Lh^{d+2}$.  
Consequently, the $\ll \cdot, \cdot \rr_{h}$-orthogonal projection $\Pihg$ is well-defined, cf.~\eqref{def:Pi-2}, and the decomposition of functions in $\Lh^{d+2}$ into $\Ahg$ and the $\ll \cdot,\cdot \rr_{h}$-orthogonal complement $\Ahg^\perp$ is unique. 
The following result highlights the link between the adjointness of the differential operators in $H^1_{\Gc}(\Omega)$ and $H_{\Gamma}(\divergence;\Omega)$ and the orthogonality of the spaces $\Ahg$ and $\Ahg^\perp$.

\begin{lem}[Characterization of $\Ahg^\perp$]\label{lem:char-Ah0perp}
Let \hypref{hyp:B-h-dt} be satisfied on an open bounded set $\Omega\subset \R^d$ with Lipschitz boundary $\partial \Omega$. 
Let $\Gamma \subset \partial \Omega$ be a relatively open subset with finitely many connected components. 

Then, there is a one-to-one correspondence between $\Phi \in \Ahg^\perp$ and a pair of functions $\left(\q,q_B\right)\in L^2(\Omega)^2$ with $h\q \in H^1_{\Gc}(\Omega)$ such that 
$\Psi_h\left(\q,q_B\right) = \Phi$. 
\end{lem}
\begin{proof}
First note that the proof of \propref{prop:reg-Ah}.\ref{itm:reg-Ah-q} is still valid if we replace $\Ahb$ by $\Ahg$. 
This implies that for any $\Phi \in \Ahg^\perp$ there exist there exists a  unique pair of functions $\left(\q,q_B\right)\in L^2(\Omega)^2$ such that $\Psi_h\left(\q,q_B\right)=\Phi$ and additionally $h\q\in H^1(\Omega)$. 
It remains to show that $h\q |_{\Gc} = 0$ in the sense of $H^{1/2}$-traces. 
Starting from~\eqref{eq:reg-Ah}, integrating by parts with~\eqref{eq:ibp-dual} 
we obtain in particular for any smooth function $V = (V_1,V_2,V_3)^\top \in\Ahg$ that
\begin{align*}
0 &= 
\ll V_1, h\q \cdot \nu \rr_{H^{-1/2}(\partial \Omega), H^{1/2}(\partial \Omega)} 
= \int_{\Gamma}   h\q V_1\cdot\nu \, \d s(x)  
 +\int_{\Gc}   h\q V_1\cdot\nu \,\d s(x),  
\end{align*}
and the boundary term on $\Gamma$ vanishes since $V \in \Ahg$. 
Since $\Gc$ is a Lipschitz curve with finitely many connected components we conclude that $hq_1|_{\Gc}=0$ in the sense of $H^{1/2}$-traces. 

The reverse implication follows combining the arguments in \lemref{lem:proj-Rn} and~\eqref{eq:ibp-dual} leading to
\begin{equation*}
\int_{\Omega} v \cdot \nabla \phi \dx + \int_{\Omega} \diver v \, \phi \dx = 0,
\end{equation*}
for any $\phi\in H^1_{\Gc}(\Omega)$ and any $v\in H_{\Gamma}(\divergence;\Omega)$.
\end{proof}

Now we are in the position to deduce the following well-posedness result. 

\begin{prop}[Homogeneous boundary conditions]
\label{prop:proj-dt-bd-hom}
Let \hypref{hyp:B-h-dt} and \hypref{hyp:dom-dt} be satisfied with a bounded open set~~$\Omega \subset \R^d$. 
 
For any function $U^* \in \Lh^{d+2}$ there exists a unique solution of the correction step~\eqref{eq:GN-dt-proj} on $\Omega$ subject to homogeneous boundary conditions~\eqref{eq:bc-dt-hom}, consisting of functions $U_0\in\Ahgu$ and $\left(\q_0,q_{B0}\right)\in L^2(\Omega)^2$ with $ h \q_0 \in H^1_{\Gamma_{hq}}(\Omega)$.  
The solution is given by
\begin{equation*}
U_0 = \Pi_h[\mathbb{A}_{h,\Gamma_u}] (U^*)
\qquad\textrm{and}\qquad
\left(\q_0,q_{B0}\right)=\Psi_h^{-1}\left(\frac{U^*-U_0}{\dt}\right).
\end{equation*}
\end{prop}
As a consequence of the projection structure, the energy conservation law~\eqref{eq:energy} still holds on the bounded domain with homogeneous boundary condition~\eqref{eq:bc-dt-hom}.

\subsubsection{Inhomogeneous boundary conditions}\label{sec:dt-bd-inhom}

As mentioned before, to consider inhomogeneous boundary conditions as in~\eqref{eq:bc-dt-inhom}, we have to give a sense to the restriction of distributions in $H^{-1/2}\left(\partial \Omega\right)$ to a part of the boundary. 
Due to the non-locality of such distributions, this has to be done in a weak sense.  
For this purpose let us introduce the equivalence relation that appears also in the definition of $H_{\Gamma}(\divergence;\Omega)$ above. 
For any $\widetilde v,\bcu \in H^{-1/2}\left(\partial \Omega\right)$, we define
\begin{equation}\label{equivalence}
\widetilde v\Gsim\bcu
\quad\textrm{iff}\quad
\ll \widetilde v ,f\rr_{H^{-1/2}(\partial \Omega), H^{1/2}(\partial \Omega)} = \ll \bcu, f \rr_{H^{-1/2}(\partial \Omega), H^{1/2}(\partial \Omega)}
\quad\text{for all}\;\;
f \in H^1_{\Gc}(\Omega).
\end{equation} 
We shall impose data on the normal trace in $H^{-1/2}(\partial \Omega)$ in the sense of this equivalence relation. 
This means in particular that only the equivalence class of the given boundary datum $\bcu$ is used. 
However, note that this way of imposing boundary values is very weak. In fact, imposing more regular data $\bcu$ does not lead to more regular normal traces in general. 
Even if one assumes that $\bcu \in H^{1/2}(\partial \Omega)$, the solution might not be sufficiently regular to ensure integrability of the normal trace. 
Only if $u$ is sufficiently regular, e.g., if $u \in H^1(\Omega)$, we know that the normal trace of $u$ agrees with $\bcu$ a.e. on $\Gamma$.

This ambiguity would already appear for more classical equations such as the incompressible Euler equations. 
Thus, we shall not elaborate on this here.

\begin{defi}[Admissible functions with inhomogeneous trace condition]\label{def:Ah-Pi-bd-dt}
Let \hypref{hyp:B-h-dt} and \hypref{hyp:dom-dt}  be satisfied. 
For any function $\bcu\in H^{-1/2}(\partial\Omega)$ we denote the set of admissible functions with 
normal trace equivalent to $\bcu$ on $\Gamma_u$ in the sense of~\eqref{equivalence} by
\begin{equation*}
\mAhguu \coloneqq \left\{V=\left(V_1,V_2,V_3\right)^\top \in \Ahb
\colon\quad
V_1 \cdot \nu\Gusim\bcu
\right\}.
\end{equation*}
\end{defi}
Note that $\mAhgu (0) = \mathbb{A}_{h,\Gamma_u}$. 
For $\bcu \neq 0$ the set $\mAhguu$ is not a linear, but an affine closed subspace.  
Hence, an affine projection mapping to $\mAhguu$ can be defined.

One may also include inhomogeneous boundary conditions on $h\q$.
For brevity of the notation we introduce the affine closed subset of functions in $H^1(\Omega)$ with given trace $\bchq \in H^{1/2}(\Gamma_{hq})$ by 
\begin{equation*}
\mHghq(\bchq)
\coloneqq \left\{f \in H^1(\Omega) \colon f|_{\Gamma_{hq}} = \bchq\;\; \text{ in the sense of } H^{1/2}\text{-traces}\right\},
\end{equation*}
and note that $\mHghq(0) = H^1_{\Gamma_{hq}}(\Omega)$. 
We apply the standard approach using reference functions to the problem to reduce the inhomogeneous case for both velocity and pressure functions to the homogeneous one considered in~\secref{sec:dt-bd-hom}. 
More precisely, for given $\bcu \in H^{-1/2}(\partial \Omega)$ and $\bchq \in H^{1/2}(\Gamma_{hq})$ we call any pair of functions $\Ur\in \mAhguu$ and $\qr\in L^2(\Omega)$ such that 
~~$h \qr \in \mHghq(\bchq)$ {\it reference functions}. 
Note that the reference functions do not have to satisfy a system of equations, but satisfy the given boundary values as specified.
The existence of such reference functions can be proved by classical extension results, see~\cite[Ch.~I]{GR.1986}.
Observe that one can reduce the problem as follows: 
Let the functions $U \in \mAhgu(\bcu)$ and $\left(\q,q_B\right)\in L^2(\Omega)^2$ such that $h\q \in \mHghq(\bchq)$ be solutions of the correction step~\eqref{eq:GN-dt-proj} subject to the inhomogeneous boundary conditions~\eqref{eq:bc-dt-inhom}. Then, setting 
\begin{equation*}
U_0=U-\Ur 
\qquad\textrm{and}\qquad
\q_0=\q-\qr \in L^2(\Omega),
\end{equation*}
 we find that $U_0\in \mathbb{A}_{h,\Gamma_u}$ and $\q_0 \in L^2(\Omega)$ such that $h \q_0 \in H^1_{\Gamma_{hq}}(\Omega)$.  
 Furthermore, by linearity of the equations we obtain that $U_0, \q_0$ 
are solutions of some system of equations depending on $\Ur, \qr$ subject to the homogeneous boundary conditions~\eqref{eq:bc-dt-hom}. 
For this system of equations subject to homogeneous boundary conditions \propref{prop:proj-dt-bd-hom} can be applied. 
With this strategy for the case of inhomogeneous boundary conditions we obtain well-posedness of the correction step on bounded domains. 

\begin{thm}[Inhomogeneous boundary conditions]\label{thm:proj-dt-bd-inhom}
Let \hypref{hyp:B-h-dt} and \hypref{hyp:dom-dt} be satisfied with a bounded open set~~$\Omega \subset \R^d$.   

For any given functions $U^* \in \Lh^{d+2}$, $\bcu\in H^{-1/2}(\partial \Omega)$ and $\bchq\in H^{1/2}(\Gamma_{hq})$
there exists a unique solution of the correction step~\eqref{eq:GN-dt-proj} on~~$\Omega$ subject to inhomogeneous boundary conditions~\eqref{eq:bc-dt-inhom}, 
consisting of functions $U\in\mAhgu(\bcu)$ and $\left(\q,q_{B}\right)\in L^2(\Omega)^2$ with $ h \q \in  \mHghq(\bchq)$.  
The solution is given by
\begin{equation}\label{eq:dt-sol-inh-1}
U=\Ur+\Pi_h[\mathbb{A}_{h,\Gamma_u}]\left(U^*-\Ur-\dt\Psi_h\left(\qr,0\right)\right)
\quad\textrm{and}\quad
\left(\q,q_{B}\right) = \Psi_h^{-1}\left(\frac{U^*-U}{\dt}\right),
\end{equation}
for any pair of reference functions $\Ur\in\mAhgu(\bcu)$ and~~$\qr\in L^2(\Omega)$ such that $h \qr \in \mHghq(\bchq)$. 
\end{thm}
\begin{proof}
We start by showing that the functions given by~\eqref{eq:dt-sol-inh-1} are a solution of the correction step~\eqref{eq:GN-dt-proj} with inhomogeneous boundary conditions~\eqref{eq:bc-dt-inhom}. 
By replacing the unknowns $U = \Ur + U_0 \in \mAhgu(\bcu)$ and $\q = \qr + \q_0\in L^2(\Omega)$ with $h\q \in \mHghq(\bchq)$ let us rewrite~\eqref{eq:GN-dt-proj} as
\begin{equation*}
U^* = U + \dt \Psi_h\left( \q, q_B \right) 
=  U_0 + \Ur +  \dt \Psi_h\left( \qr + \q_0, q_B \right).
\end{equation*} 
The new unknown functions are $U_0 \in \Ahgu$ and $\q_0 \in L^2(\Omega)$ such that $h \q_0 \in H^1_{\Gamma_{hq}}(\Omega)$.
Thus, equivalently to solving the correction step with inhomogeneous boundary conditions, we want to find $U_0 \in \Ahgu$, and $\q_0, q_B$ with $h \q_0 \in H^1_{\Gamma_{hq}}(\Omega)$ such that
\begin{equation*}
U_0 = U^* - \Ur - \dt \Psi_h  \left( \qr , 0\right) - \dt \Psi_h \left(\q_0, q_B \right), 
\end{equation*}
where we also have used the bilinearity of $\Psi_h$.
With $U^* - \Ur - \dt \Psi_h  \left( \qr , 0\right) \in \Lh^{d+1}$, \propref{prop:proj-dt-bd-hom} ensures the existence of such solutions, given by 
\begin{align*}
U_0 
&= \Pi_h[\mathbb{A}_{h,\Gamma_u}]
\left(U^*-\Ur-\dt\Psi_h\left(\qr,0\right)\right),\\
\left(\q_0,q_{B} \right)
&= \Psi_{h}^{-1} \left(\frac{ U^*-\Ur-\dt\Psi_h\left(\qr,0\right)- U_0}{\dt}\right)
=  \Psi_{h}^{-1} \left(\frac{ U^*-U}{\dt}\right) - \left(\qr,0\right),
\end{align*}
where again bilinearity of $\Psi_h$ is used. 
We conclude that $U,\q,q_B$ given by~\eqref{eq:dt-sol-inh-1} satisfy~\eqref{eq:GN-dt-proj}. 

It remains to show uniqueness of solutions to the correction step~\eqref{eq:GN-dt-proj} with inhomogeneous boundary conditions~\eqref{eq:bc-dt-inhom}, which shows in particular that the solutions are independent of the reference functions $\Ur, \qr$.  
Assume that for given $U^* \in \Lh^{d+2}$ and boundary data $\bcu, \bchq$ as before there are two solutions to the problem~\eqref{eq:GN-dt-proj} subject to the inhomogeneous boundary conditions.  
Due to the linearity of the equations, the  difference of the solutions satisfies the system of equations~\eqref{eq:GN-dt-proj} for $U^* = 0$ and homogeneous boundary conditions. 
By \propref{prop:proj-dt-bd-hom} we have uniqueness of solutions, and since the trivial functions is a solution it follows that the two solutions agree. 
\end{proof}

For inhomogeneous boundary conditions the energy identity~\eqref{eq:energy} is not satisfied because orthogonality is lost through the affine shift by reference functions. 
Indeed, in the scalar product $\ll U, \Psi_h\left(\q, q_B \right) \rr_h$ the boundary term does not vanish.  
Thus, we arrive at the following energy identity
\begin{align*}\label{eq:energy_b}
\frac{1}{\dt} \left(
\norm{U}_h^2 - \norm{U^*}_h^2 \right)
&= 
- \dt \norm{\Psi_h\left(\q,q_B\right)}_h^2
  - 2 \ll U, \Psi_h\left(\q,q_B\right)\rr_{h}
\\&= 
- \dt\norm{\Psi_h\left(\q,q_B\right)}_h^2
 - 2\ll \u \cdot \nu , h \q \rr_{H^{-1/2}(\partial \Omega),H^{1/2}(\partial \Omega)}
\\&= 
- \dt\norm{\Psi_h\left(\q,q_B\right)}_h^2
  - 2
  \left( \int_{\Gamma_u}\bcu\, h\q\,\d s(x)
+\int_{\Gamma_{hq}} \left(\u\cdot\nu\right)\, \bchq\,\d s(x)
\right),
\end{align*}
where the last equality holds only if the traces are sufficiently regular.

\begin{rmk}Let us conclude this section with further remarks on the boundary conditions. 
\begin{enumerate}[label=(\roman*),leftmargin=0cm]\setlength{\itemindent}{1cm}
\item (Projection property)
Instead of including inhomogeneous boundary conditions  on the velocity by means of reference functions one might directly work with the affine projection $\widetilde{\Pi}_h[\mAhgu\left(\bcu\right)] $ mapping to the closed affine space $\mAhgu(\bcu)$. 
This is defined by 
\begin{equation*}
\norm{\widetilde{\Pi}_h[\mAhgu(\bcu)]\left(U^*\right) - U^*}_h^2 = \min_{V \in \mAhgu(\bcu)} \norm{V - U^*}_h^2 \quad \text{ for any } U^* \in \Lh^{d+2}. 
\end{equation*}
However, by this approach we do not gain the same insight into the boundary values of $h\q|_{\Gamma_{hq}}$ and cannot immediately include inhomogeneous boundary conditions on the pressure. 
This, in turn, will be of practical interest in~\secref{sec:fully-discr-bd}. 
In fact, one can show that the mapping $P_h\colon U^* \mapsto U$ given by \theoref{thm:proj-dt-bd-inhom} is a projection in the sense that  $P_h|_{\mAhgu(\bcu)}$ is the identity if and only if~~$\bchq = 0$ or $\Gamma_{hq} = \emptyset$. 
In this case the mapping $P_h$ coincides with the affine projection $\widetilde{\Pi}_h$. 

\item (Well-posedness) 
The fact that well-posedness for solutions with $U \in \mAhgu(\bcu)$ and $h\q \in \mHghq(\bchq)$ is guaranteed means that in this framework one cannot hope to impose anything extra in terms of boundary conditions. 
In particular, once $\Gamma_u$ and $\Gamma_{hq}$ are fixed as in \hypref{hyp:dom-dt} and once the values for $\bcu$ and $\bchq$ are prescribed, one cannot prescribe anything extra on $\u \cdot \nu |_{\Gamma_{hq}}$ or on $h\q |_{\Gamma_u}$ in a suitable sense. 
The freedom consists in the possibility to choose the decomposition of $\partial \Omega$ and the data imposed on one of the functions on each of the parts of the boundary. 
\item (Elliptic equation)
One can show that $U = \left(\u,\w,\sigma\right)^\top$ and $\left(\q,q_B\right)$ form a solution to the corrections step as discussed above if and only if~~$\u$ is a unique solution of an elliptic system of equations, cf.~\cite{Parisot19}. 
If the functions involved are sufficiently regular for the traces of the following terms to be well-defined, one can see that $\u$ satisfies on $\Gamma_{hq}$ 
\begin{equation*}
\frac{h^2}{2}\nabla B\cdot\u-\frac{h^3}{3}\diver\u 
 =
\frac{h^2}{2} \(\w^*+\frac{\sigma^*}{\sqrt{3}}\) + \dt \bchq
\qquad\textrm{on }\Gamma_{hq}.
\end{equation*}
Similarly, under suitable assumptions on the solutions equivalence to $ h \q $ being a solution to an elliptic equation (see~\cite{SainteMarie16} in a similar context) can be proved.  
Again under suitable regularity assumptions it follows that
$h\q$ satisfies on $\Gamma_{u}$ the following relation
\begin{align*}
\partial_{\nu}\left(h\q\right) +\frac{\partial_\nu B}{4+\abs{\nabla B}^2} & \left( 6 \q-\nabla B\cdot\nabla\left(h \q\right) \right)
\\
& =
\frac{h}{\dt} \left(
 \u^* \cdot \nu+ \frac{\partial_\nu B}{4+\abs{\nabla B}^2} \left( \w^*-\sqrt{3} \sigma^*- \u^* \cdot \nabla B\right)
  - \bcu
\right)
 \quad \text{ on }\;\; \Gamma_u,
\end{align*}
where $\partial_\nu f \coloneqq \nabla f \cdot \nu$. 
This shows that there is no freedom to impose additional conditions. 
\item (Alternative boundary conditions) Our study is based on the choice to define the solution of the time-discrete Green--Naghdi model as the projection of the solution to the shallow water step onto a set of admissible functions. 
On bounded domains alternative formulations of the problem might lead to other (well-posed) boundary conditions. 
\end{enumerate}
\end{rmk}

\section{Projection scheme for the fully discrete correction step}
\label{sec:fully-discr-bd}
In this section we consider the correction step for the fully discrete problem, i.e., additionally to the time discretization as introduced in~\secref{sec:time-discrete} we also discretize in space.
Assume that a numerical scheme  is given that approximates the advection step on a polygonal tesselation $\T$ of a spatial domain $\Omega \subset \mathbb{R}^d$.
Let $\dofh$, $\dofu$, $\dofw$ and $\dofsig$ denote the respective sets of degrees of freedom of the unknown functions  $h$, $\u$, $\w$ and $\sigma$. 
For collocated schemes one has that $\dofh=\dofw=\dofsig=\T$ and $\dofu = \left(\T\right)^d$, see~\cite{Bouchut04}. 
For staggered discretizations one still has that $\dofh=\dofw=\dofsig = \T$, but in general $\dofu \neq \left(\T\right)^d$, see~\cite{Herbin14}.
 Even more generally there are schemes for which the components of the horizontal velocity $\u$ do not use the same degrees of freedom, see~\cite{Harlow65}.
For high order numerical methods such as finite element methods or discontinuous Galerkin methods, for the description of an unknown function $a$ several degrees of freedom on each cell $k \in \T$ of the mesh are required so that $\card\left(\dofa\right)>\card\left(\T\right)$, cf.~\cite{DiPietro12}.
Schemes as mentioned lead to a non-negative water depth $h_\star \coloneqq \left(h_k\right)_{k\in\dofh} \geq 0$ and the velocity function $U_\star^*=\left(\left(\u_k^*\right)_{k\in\dofu}, \left(\w_k^*\right)_{k\in\dofw}, \left(\sigma_k^*\right)_{k\in\dofsig}\right)$
as approximate solution of the advection step~\eqref{eq:GN-dt-adv}. 
Note that in the advection step also the time step $\delta_t>0$ is fixed.  

Now we shall focus on the numerical approximation of the correction step. 
In~\secref{NumWholeSpace} we present a strategy to construct a numerical scheme satisfying a discrete projection property on the whole space without dry areas. Hence, there is no boundary. 
This represents a discrete counterpart of the situation in \lemref{lem:proj-Rn}.
Implementing boundary conditions is not straightforward for general fully discrete schemes. 
Thus, in~\secref{NumSimple} we choose one specific scheme to be considered from then on. 
Then, in~\secref{NumBounded} we focus on the numerical boundary conditions for bounded spatial domains.
As in the space-continuous situation in~\secref{sec:proj-dt-bd} we identify a class of boundary conditions that are compatible with a discrete projection property of the scheme.
In~\secref{NumBC_bound} we start by considering faces on the boundary of the computational domain.
Then in~\secref{NumBC_int} we focus on dry front faces.

\subsection{Whole space domain}\label{NumWholeSpace}

In the current section we present an approach for the whole space domain, which is adapted in the sequel to identify a class of boundary conditions for bounded domains. 

\subsubsection{General strategy}\label{NumOverview}
In this section we present a general approach to construct schemes with a discrete orthogonal projection property.  
This strategy is applicable in numerous situations besides the ones considered here.  
We present three steps that, by construction, lead to a numerical scheme which is a projection onto a certain discrete linear space. 
\begin{enumerate}[label=(\Roman*),leftmargin=0cm]\setlength{\itemindent}{1cm}
\item{\style Discrete scalar product:}
We require a discrete weighted scalar product.
For given positive $H_\star = \left(H_k\right)_{k\in\dofh}>0$, and each of the unknown functions $a \in \left\{\u,\w,\sigma\right\}$ we assume that a scalar product 
$\ll \cdot,\cdot\rr_{H_\star}^{\delta,a} \colon \R^{\card(\dofa)} \times \R^{\card(\dofa)} \to \R$ is specified.
Then, a discrete counterpart of the weighted Lebesgue space for one of the unknown functions $a$ is given by the weighted space of sequences
\begin{equation*}
\ell^2\left(\dofa;H_\star\right)
=\left\{\chi_\star=\left(\chi_k\right)_{k\in\dofa}\colon \norm{\chi_\star}_{H_\star}^{\delta,a}<\infty \right\}
\quad\text{with}\quad
\norm{\chi_\star}_{H_\star}^{\delta,a}\coloneqq\sqrt{\ll \chi_\star,\chi_\star\rr_{H_\star}^{\delta,a}}.
\end{equation*}
We denote the degrees of freedom for the full velocity vector $U = \left(\u,\w,\sigma\right)$ with slight abuse of notation by $\dofU \coloneqq \dofu \times \dofw \times \dofsig$.
Then, for the full vector of unknown functions the weighted discrete scalar product $\ll\cdot,\cdot\rr_{H_\star}^{\delta}\colon \R^{\card(\dofU)} \times \R^{\card(\dofU)}\to\R$ is given by
\begin{equation*}
\ll\begin{pmatrix} U_{1\star}\\U_{2\star}\\U_{3\star}\end{pmatrix},\begin{pmatrix}V_{1\star}\\V_{2\star}\\V_{3\star}\end{pmatrix}\rr_{H_\star}^{\delta} \coloneqq \ll U_{1\star},V_{1\star}\rr_{H_\star}^{\delta,\u}
+\ll U_{2\star},V_{2\star}\rr_{H_\star}^{\delta,\w}
+\ll U_{3\star},V_{3\star}\rr_{H_\star}^{\delta,\sigma}.
\end{equation*}
It equips the space $\ell^2\left(\dofU;H_\star\right) \coloneqq  \ell^2\left(\dofu;H_\star\right)\times\ell^2\left(\dofw;H_\star\right)\times\ell^2\left(\dofsig;H_\star\right)$ with a scalar product. 
Note that this shall be applied for the water depth function $h_\star$ given by the advection step, assuming for now that $h_k >0$, for any $k \in \dofh$. 
 Later in~\secref{NumBC_int} the interior dry front shall be treated as part of the boundary. 
\item{\style Admissible discrete functions:}
We consider a discrete version $\Ahd$ of the space of admissible functions $\Ahb$. 
To avoid a restriction of the presentation to a particular scheme, at this point we keep the definition of $\Ahd$ relatively vague. 
We only demand that it is a closed linear subspace of $\ell^2(
\dofU; h_{\star})$ that is consistent with the definition of $\Ahb$. 
Then we denote 
by $\Pi_{h_\star}^\delta[\Ahd]$ the $\ll \cdot, \cdot \rr_{h_\star}^\delta$-orthogonal projection mapping to $\Ahd$.
Motivated by \lemref{lem:proj-Rn} we define the approximate velocity vector for the fully discrete problem on the whole space as the projection to the discrete space of admissible function $\Ahd$, that is, by
\begin{equation*}\label{NumProj}
U_\star=\Pi_{h_\star}^\delta[\Ahd]\left(U_\star^*\right)
\qquad\textrm{ for any }\quad
U_\star^* \in \ell^2\left(\dofU;h_\star\right).
\end{equation*}
By this $U_\star$ is fully determined and we do not have any additional freedom. 
However, it is not clear how to compute it.  
In fact there are various numerical methods to determine linear projections, such as minimization schemes or methods based on a variational formulation. 
At this stage it is also clear that there exists a uniquely determined function $\Phi_\star=\left(\left(\phi_{1k}\right)_{k\in\dofu},\left(\phi_{2k}\right)_{k\in\dofw},\(\phi_{3k}\right)_{k\in\dofsig}\right)\in\left(\Ahd\right)^\perp$ such that
\begin{equation*}
U_\star=U_\star^*-\delta_t\Phi_\star,
\end{equation*}
where $\left(\Ahd\right)^\perp$ is the $\ll\cdot,\cdot\rr^{\delta}_{h_\star}$-orthogonal complement of $\Ahd$ in $\ell^2(\dofU;h_{\star})$.  

\item{\style Discrete hydrodynamic pressures:} 
It remains to identify certain discrete pressure functions with $\Phi_\star$ so that equation~\eqref{eq:GN-dt-proj-pre} is satisfied in a suitable sense. 
If one is interested only in $U_\star$, at least for first order schemes this need not be done in the computation but only theoretically. 
Let $\dof^\q$ and $\dof^{q_B}$ denote the respective degrees of freedom of $\q$ and $q_B$. 
If there is an invertible mapping $\Psi_{h_\star}^\delta \colon \R^{\card(\dof^\q)} \times \R^{\card(\dof^{q_B})} \to \left(\Ahd\right)^\perp$, which is a consistent discrete counterpart of the mapping $\Psi_h$ defined in~\eqref{def:psi-inv}, then the discrete hydrodynamic pressure functions can be recovered by
\begin{equation*}
\left(\mat{c}{\q_k\\q_{Bk}}\right)=\left(\Psi_{h_\star}^\delta\right)^{-1}\left(\Phi_\star\right)|_{k}.
\end{equation*} 
Note that the discretization of the hydrodynamic pressure functions (determined by $\dof^\q$ and $\dof^{q_B}$) is less relevant as long as the mapping $\Psi_{h_\star}^\delta$ is invertible. 
If one chooses $\left(\Psi_{h_\star}^\delta\right)^{-1}|_{k} \left(\Phi_\star\right)\coloneqq \Psi_h^{-1} \left(\Phi_k\right)$ with $\Psi_h^{-1}$ as defined in~\eqref{def:psi-inv} then it follows that $\dof^{q_B}=\dofw$ and if $\dofw = \dofsig$, then also $\dof^\q=\dofsig  = \dofw$. 
However, in general it is not necessary to choose the same degrees of freedom for the discrete hydrodynamic pressures as for the velocity functions, as is the case for staggered grids.  
\end{enumerate}

\subsubsection{Application to a simple case}\label{NumSimple}

To showcase the benefits of the previously outlined strategy we apply it in a simple situation. 
We assume that to each polygonal cell $k\in\T$ in a given regular tesselation $\T$ of the spatial domain $\Omega$, a single degree of freedom of each unknown is assigned, i.e., $\dofh=\dofw=\dofsig=\T$ and $\dofu = \T^d$. 
For example, the degree of freedom may represent the mean value of a continuous function on the cell $k$, as for finite volume methods of lowest order, cf.~\cite{Bouchut04}.
\begin{enumerate}[label=(\Roman*),leftmargin=0cm]\setlength{\itemindent}{1cm}

\item 
For this choice of degrees of freedom several schemes for the shallow water step are entropy-satisfying in the sense that a discrete inequality of the form
\begin{align*}
 \sum_{k\in\T} h_k^{n*}\abs{ \u_k^{n*}}^2 
\le
\sum_{k\in\T}h_k^n \abs{\u_k^n}^2,
 \end{align*}
is satisfied in the $n$-th time step. 
Analogous weighted estimates can be obtained for the discrete solutions $\w^{n*}_{\star}$ and $\sigma^{n*}_{\star}$ of the advection equations~\eqref{eq:GN-dt-adv}. 
In this way we obtain a discrete version of the balance law \eqref{Entropy_adv} by use of suitable schemes. 
Note that this determines the scalar product.
More precisely, for any $a\in\l\{\u,\w,\sigma\r\}$ we set
\begin{equation*}
\ll \psi_\star,\chi_\star\rr_{h_\star}^{\delta,a}\coloneqq\sum_{k\in\W} \psi_k\cdot \chi_k\, h_k\, \m_k, 
\end{equation*}
where $\m_k$ is the $d$-dimensional volume of the cell $k \in \T$.
For vector-valued arguments the dot product denotes the canonical scalar product of vectors in $\R^d$ and otherwise the multiplication of scalars. 
Here the subset $\W\subseteq\T$ is the domain where we apply the projection.
 In fact, the domain has to be restricted to the wetted areas $\W\coloneqq \l\{k\in\T\colon h_k>0\r\}$, since the scalar product is not defined on dry areas where $h_k = 0$.
This means that a particular boundary, the so-called \textit{dry front}, has to be considered. 
We investigate this situation in~\secref{NumBC_int}. 
Here for simplicity we assume that the water depth is positive on the whole domain and thus we have $\W=\T$.  
Note that we can identify $\ell^2(\dofU;h_\star)$ with $\ell^2(\T;h_\star)^{d+2}$. 
 
\item We define the discrete admissible set by
\begin{equation}\label{NumAdmissibleSet}
\Ahd \coloneqq \left\{
U_\star=\left(\mat{c}{\u_\star\\ \w_\star\\\sigma_\star}\right)\in \Lhd^{d+2} \colon 
\quad
\mat{l}{
\w_k = \u_k \cdot \nablad_k B_\star - \frac{h_k}{2} \nablad_k\cdot \u_\star\ ,
\\
\sigma_k = - \frac{h_k}{2 \sqrt{3}} \nablad_k\cdot \u_\star \; \quad \text{ for all } k \in \T
}\right\}
\end{equation}
with the centered approximation $\nabla^{\delta}_{\star}$ and $\nabla^{\delta}_\star \cdot$ of the gradient and divergence operator, respectively, defined by
\begin{equation*}
\nablad_k\phi_\star \coloneqq\frac{1}{\m_k}\sum_{\Fd_k}\frac{\phi_k+\phi_\kf}{2}\n_k^\kf\m_f
\qquad\textrm{and}\qquad
\nablad_k\cdot\phi_\star \coloneqq\frac{1}{\m_k}\sum_{\Fd_k}\frac{\phi_k+\phi_\kf}{2}\cdot\n_k^\kf\m_f. 
\end{equation*}
Here $\Fd_k$ is the set of faces of the cell $k$, $\m_f$ is the $(d-1)$-dimensional volume of a face $f \in \Fd_k$ (with the convention $\m_f=1$ for $d=1$),
$\kf\in\T$ is the cell adjacent to $k$ sharing the face $f$ and $\n_k^\kf$ is the outward unit normal of a cell $k$ on the face $f$. 
We also introduce the discrete bathymetry $B_\star=\left(B_k\right)_{k\in\T}$. 
Then we may set $U_{\star} \coloneqq \Pi_{h_{\star}}^{\delta}[\Ahd](U^*_{\star})$ and there is a discrete function $\Phi_{\star} \in \left(\Ahd\right)^\perp$ such that
\begin{align}\label{NumDecomposition}
U_{\star} =  U^*_{\star} - \dt \Phi_{\star}. 
\end{align}
To efficiently compute approximate solutions with this scheme we identify  $\Phi_{\star} \in \left(\Ahd\right)^\perp$. 
A computation similar to the one in~\eqref{eq:proj-Rn} leads to 
\begin{equation}\label{NumOrtho}
0
=\ll V_\star,\Phi_\star\rr_{h_\star}^\delta
=
\sum_{k\in\T}\left(h_k\left(\phi_{1k}+\phi_{2k}\nablad_kB_\star\right)+\nablad_k\left(\frac{h_\star^2}2\left(\phi_{2\star}+\frac{\phi_{3\star}}{\sqrt{3}}\right)\right)\right) \cdot V_{1k} \, \m_k,
\end{equation}
for any $V_\star=\left(V_{1\star},V_{2\star},V_{3\star}\right)^\top\in\Ahd$ and any $\Phi_\star=\left(\phi_{1\star},\phi_{2\star},\phi_{3\star}\right)^\top\in\left(\Ahd\right)^\perp$.
We have used the fact that the discrete gradient is the dual operator of the discrete divergence with respect to the (unweighted) scalar product, which means that
\begin{equation*}
\ll \nablad_\star\cdot \psi_\star,\chi_\star\rr_{1}^\delta=-\ll \psi_\star,\nablad_\star \chi_\star\rr_{1}^\delta
\quad\textrm{for any}\ \psi_\star\in\Lhd^d
\textrm{ and }\chi_\star\in\Lhd.
\end{equation*}
Here $1$ denotes the constant function with value $1$.  
From~\eqref{NumOrtho} it follows that $\phi_{1\star}$ is determined by $\phi_{2\star}$ and $\phi_{3\star}$ via
\begin{align}\label{NumAdmissibleSetOrtho}
h_k \phi_{1k} =
 - \left(\nablad_k\left(\frac{h_\star^2}2\left(\phi_{2\star}+\frac{\phi_{3\star}}{\sqrt{3}}\right)\right)  +  \phi_{2k}\nablad_k B_\star\right).
\end{align}
Together, the two constraints in~\eqref{NumAdmissibleSet}, \eqref{NumDecomposition} and~\eqref{NumAdmissibleSetOrtho} form to a linear system of equations of size $(5+d)\card\left(\T\right)$ with unknowns $U_\star$ and $\Phi_\star$. 
By substitution the system can be reduced to a system of size $d\card\left(\T\right)$ for the unknown $\u_\star$ only. 
This leads to the scheme ($GN^\delta$) proposed in~\cite{Parisot19} where the correction step is computed by solving a system for the velocity.
Alternatively, one may solve the system for $\Phi_\star$, cf.~\cite{SainteMarie16}. 

\item 
Finally, a reconstruction of the hydrodynamic pressure functions is obtained by setting
\begin{equation*}
\left(\mat{c}{\q_k\\q_{Bk}}\right)
\coloneqq
\Psi_h^{-1} \left(\Phi_k\right) 
=
-h_k\left(\mat{c}{\frac{1}{2} \left(\phi_{2k}+ \frac{\phi_{3k}}{\sqrt{3}}\right)
\\\ds
\phi_{2k}}\right).
\end{equation*}
By the projection property we obtain a discrete version of \eqref{Entropy_proj}. 
\end{enumerate}
Then, combining the estimates for both steps the full scheme is entropy-satisfying. 
\subsection{Bounded spatial domain}\label{NumBounded}

In the following let us consider the case of a bounded domain $\Omega$.
As for the time-discrete case in~\secref{sec:proj-dt-bd} we modify the space of admissible functions such that the orthogonality property in~\eqref{NumOrtho} is preserved for homogeneous boundary conditions. 
More specifically, we aim for a discrete space $\Ahdg$ including boundary conditions and an invertible mapping $\Psi_{h_\star}^{\delta}$ that maps to a subset of $\ell^2(
\dofU; h_\star)^{d+2}$ such that
\begin{align}\label{eq:discr-orth}
\ll V_\star,\Phi_{\star}\rr_{h_\star}^\delta 
= 0 
\quad\text{for all}\;\; V_\star \in \Ahdg\quad\text{and all}\;\; \Phi_{\star} \in \Img(\Psi_{h_\star}^{\delta}). 
\end{align}
For the sake of readability, we illustrate the strategy using the discretization presented in~\secref{NumSimple}.
An analogous procedure can be applied for more complex schemes but suitable boundary conditions are highly dependent on the choice of discretization.
\medskip

From~\eqref{NumOrtho} it follows by a straight-forward computation that for $U_\star = \left(\u_{\star},\w_{\star},\sigma_{\star}\right)^\top \in \Ahd$ and $\Phi_\star = \Psi_{h_{\star}}^\delta\left(\q_\star,q_{B\star}\right)$ one has that
\begin{equation}\label{eq:discr-proj} 
\ll U_\star,\Phi_\star\rr_{h_\star}^\delta
=-\sum_{k\in\W}\sum_{f\in\Fd_k}\frac{h_\kf\q_\kf\u_k+h_k\q_k\u_\kf}2\cdot\n_k^\kf\m_f.
\end{equation}
Thanks to the antisymmetry of the normal $\n_k^\kf=-\n_\kf^k$, the terms on interior faces vanish and only the ones on the boundary of the domain have to be taken into account. 

From now on let $\partial\W$ denote the subset of faces in $\T$ that lie on the boundary of the domain $\W$. 
Here the cells $\ki = \ki(f)\in\W$ and $\kg = \kg(f)\not\in\W$ are adjacent to a face $f\in\partial\W$. 
The cell $\kg$ is a so-called {\it ghost cell} that can be defined for a boundary face $f\in\partial\W$ by extension of the tesselation. 
We denote by $\mathbb{G}$ the set of all such ghost cells.

Since the terms on interior faces in~\eqref{eq:discr-proj} vanish, the orthogonality in~\eqref{eq:discr-orth} holds if and only if the following condition on boundary faces holds
\begin{equation}\label{NumProjCond}
\left(h_\kg\q_\kg\u_\ki+h_\ki\q_\ki\u_\kg\right)\cdot\n_\ki^\kg = 0
\qquad\text{for any } f\in\partial\W. 
\end{equation}
Once the values $h_\kg\q_\kg$ and $\u_\kg\cdot\n_\ki^\kg$ are available, the system of equations is closed and we are in the position to apply the projection scheme presented in~\secref{NumSimple} on $\W$. 
Note that the properties of the ghost cells do not affect the scheme. 
In the following we specify -- for several different types of boundary conditions -- the values of the unknown functions on the ghost cells depending on the interior cell values such that~\eqref{NumProjCond} is satisfied. 
\medskip

Similarly as in \hypref{hyp:dom-dt} we divide the set of boundary faces  into subsets corresponding to the type of boundary condition imposed.
First, we define the faces at the (interior) dry front 
forming the set $\Gamma_h \coloneqq\l\{f\in\partial\W\colon h_{\kg}=0\r\}=\partial\W\backslash\partial\T$. 
Then, the remaining boundary faces $\partial\W\cap\partial\T$ are decomposed into two sets. 
\begin{hyp}[Decomposition of the boundary faces $\partial\W\cap\partial\T$]\label{hyp:NumBound}
Assume that there exist two sets of faces $\Gammad_u, \Gammad_{hq} \subset \partial\W\cap\partial\T$ decomposing the set of boundary faces $\partial\W\cap\partial\T$, i.e., we have that $\partial\W\cap\partial\T =\Gammad_u\cup\Gammad_{hq}$ and $\Gammad_{u}\cap\Gammad_{hq}=\emptyset$.
\end{hyp}
Similarly as in the time-discrete framework~\eqref{eq:bc-dt-inhom}, we want to impose  boundary conditions on the normal velocity $\u \cdot \nu$ on the faces in $\Gammad_u$ and on $h\q$ on faces in $\Gammad_{hq}$. 
More specifically, for given discrete real-valued boundary data
$\bcu_\star=\left(\bcu_f\right)_{f\in\Gammad_u}$ and $\bchq_\star=\left(\bchq_f\right)_{f\in\Gammad_{hq}}$,
 we set
\begin{equation}\label{NumBcInhom}
\mat{l@{\qquad}r@{~}c@{~}l@{\qquad}l}{&\ds
\left(\alpha_f\u_\ki+\left(1-\alpha_f\right)\u_\kg\right)\cdot \n_\ki^\kg
&=&\ds
\bcu_f
&\textrm{for any }f\in\Gammad_u,
\\&\ds
\alpha_f h_\ki\q_\ki+\left(1-\alpha_f\right) h_\kg\q_\kg
&=&\ds
\bchq_f
&\textrm{for any }f\in\Gammad_{hq},
}
\end{equation}
for face weights $\alpha_f \in [0,1)$ to be chosen.
The weights $\alpha_f$ determine the location in which the given values $\bcu_f$ or $\bchq_f$ are imposed.  
For example, for $\alpha_f=0.5$ the given value can be considered to be prescribed at the face, while for $\alpha_f=0$ the given value is imposed directly on the ghost cell. 
We shall see that this parameter allows us to deal in a similar manner with both boundary conditions on the boundary of the computational domain $\partial\W\cap\partial\T$ 
and with conditions on interior faces $\partial\W\setminus\partial\T$ such as the dry front.

\subsubsection{General treatment of boundary conditions}\label{NumBC_bound}
Let us start considering the faces on the boundary of the computational domain, divided into $\Gammad_u$ and $\Gammad_{hq}$. 
Those are the faces on which we impose given data. 
Analogously to the time-discrete case in \theoref{thm:proj-dt-bd-inhom}, we work with discrete reference functions to deal with inhomogeneous boundary conditions. 
Denoting by $\G$ the set of ghost cells we require that the reference functions $\ur_\star=\left(\ur_k\right)_{k\in\W\cup\G}$ and $\qr_\star=\left(\qr_k\right)_{k\in\W\cup\G}$ are discrete functions satisfying the boundary conditions~\eqref{NumBcInhom}. 
Computing such discrete reference functions may be costly.
 Hence, we aim for a class of discrete reference functions that need not be computed explicitly. 
 As in the time-discrete case we want the differences $\u_\star-\ur_\star$ and $\q_\star-\qr_\star$ to 
satisfy the homogeneous condition in~\eqref{NumBcInhom} and the projection condition in~\eqref{NumProjCond}. 
This leads to 
\begin{alignat*}{1}
\u_\kg \cdot \n_\ki^\kg
&=
\begin{cases} 
\frac{\left(1-\alpha_f\right)\ur_\kg-\alpha_f\left(\u_\ki-\ur_\ki\right)}{1-\alpha_f}\cdot \n_\ki^\kg
\quad\quad   
&\textrm{for any }f\in\Gammad_u,\\
\frac{\left(1-\alpha_f\right)\ur_\kg+\alpha_f\left(\u_\ki -\ur_\ki \right)}{1-\alpha_f}\cdot \n_\ki^\kg &\textrm{for any }f\in\Gammad_{hq},
\end{cases} \\
h_\kg\q_\kg
&=
\begin{cases}
\frac{\left(1-\alpha_f\right)h_\kg\qr_\kg+\alpha_fh_\ki\left(\q_\ki-\qr_\ki\right)}{1-\alpha_f}
&\;\textrm{  for any }f\in\Gammad_u,\\
\frac{\left(1-\alpha_f\right)h_\kg\qr_\kg-\alpha_fh_\ki\left(\q_\ki-\qr_\ki\right)}{1-\alpha_f}
&\;\textrm{  for any }f\in\Gammad_{hq}.
\end{cases}
\end{alignat*}
Since the discrete reference functions satisfy~\eqref{NumBcInhom}, the terms   $\ur_\kg \cdot \n_\ki^\kg$ on $\Gammad_u$ and $h_\kg\qr_\kg$ on $\Gammad_{hq}$ can be replaced. 
Furthermore, we choose 
$\ur_\kg \cdot \n_\ki^\kg$ on $\Gammad_{hq}$, and $h_\kg\qr_\kg$ on $\Gammad_{u}$, respectively, such that the remaining terms that depend on the reference functions disappear as well. 
Thus, the boundary conditions can be formulated independently of the reference functions and no explicit knowledge of the discrete reference function is required for the computation. 

Overall, this is achieved if we set for the discrete reference functions
\begin{equation*}%\label{NumBcRef}
  \begin{alignedat}{2}
\left(\alpha_f\ur_\ki+\left(1-\alpha_f\right)\ur_\kg\right)\cdot \n_\ki^\kg
&= 
\bcu_f \qquad 
&&\textrm{ for any }f\in\Gammad_u,\\ 
\left(\alpha_f\ur_\ki-\left(1-\alpha_f\right)\ur_\kg\right)\cdot \n_\ki^\kg
&=
0
&&\textrm{ for any }f\in\Gammad_{hq},\\ 
\alpha_f h_\ki\qr_\ki-\left(1-\alpha_f\right)h_\kg\qr_\kg
&= 
0
&&\textrm{ for any }f\in\Gammad_u
,
\\ 
%\textrm{and}\qquad
\alpha_f h_\ki\qr_\ki+\left(1-\alpha_f\right)h_\kg\qr_\kg
&= 
\bchq_f
&&\textrm{ for any }f\in\Gammad_{hq}. 
  \end{alignedat}
\end{equation*}
Existence of such reference functions is straightforward, since the problem is finite-dimensional. 

\subsubsection{Dry front condition}\label{NumBC_int}

Finally, we consider the case of the boundary of the projection domain $\W$ located in the interior of the computational domain $\T$. 
This occurs in the presence of dry areas $\l\{k\in\T\colon h_k=0\r\}$, where the projection can not be defined. 
Note that in dry areas the solution of the shallow water model and hence of the advection step is not fully determined~\cite{Lannes18}. Still, this case can be handled by several numerical schemes, cf.~\cite{BNL.2011, Bouchut04}. 

By definition of the wetted domain $\W$ for any face $f\in\partial\W\setminus\partial\T$ the water depth vanishes on the adjacent ghost cell by $h_\kg=0$. 
 Thus, it is natural to impose on the ghost cell that $h_\kg\q_\kg=0$, which is included in~\eqref{NumBcInhom} by setting $\alpha_f=0$ and $\bchq_f=0$ for any $f\in\Gammad_h$. 
The remaining ghost cell values are then determined by~\eqref{NumProjCond} as $\u_\kg \cdot \n_\ki^\kg = 0$. 
Note that this represents a boundary condition without exchange of energy.  

\paragraph*{Summary:} Altogether, we propose the following values for the ghost cells
\begin{equation}\label{NumBc}
\mat{l@{\qquad}r@{~}c@{~}l}{&\ds
\u_\kg \cdot \n_\ki^\kg
&=&
\ds\l\{\mat{l@{\qquad}l}{\ds
2\bcu_f-\u_\ki\cdot \n_\ki^\kg&\textrm{ for any }f\in\Gammad_u,\\\ds
\u_\ki\cdot \n_\ki^\kg&\textrm{ for any }f\in\Gammad_{hq},\\\ds
0&\textrm{ for any }f\in\Gammad_{h},
}\r.
\\&\ds 
\textrm{and}\qquad h_\kg\q_\kg
&=&
\ds\l\{\mat{l@{\qquad}l}{\ds
h_\ki\q_\ki&\textrm{ for any }f\in\Gammad_u,\\\ds
2\bchq_f-h_\ki\q_\ki&\textrm{ for any }f\in\Gammad_{hq},\\\ds
0&\textrm{ for any }f\in\Gammad_{h}.
}\r.
}
\end{equation}
This results from setting $\alpha_f=\frac12$ for any $f\in\Gammad_u\cap\Gammad_{hq}$. 
It is motivated by the conservation of mass when dealing with wall boundary conditions for the full model, see~\secref{BC_wall}.

\section{Numerical strategy for the full Green--Naghdi system}\label{sec:full-num}

In the preceding section the numerical strategy of the correction step is described. 
While the correction step is the  main focus of this work we shall also present a numerical scheme for the full Green--Naghdi model~\eqref{eq:GN} in the following.  
For certain typical situations in 1D we present numerical evidence to highlight the strength of the framework especially with respect to non-standard boundary conditions.
As in~\secref{sec:splitting} we use a time splitting that can be interpreted as an ImEx (implicit-explicit) scheme.

\begin{enumerate}[label=(\Roman*),leftmargin=0cm]\setlength{\itemindent}{1cm}
\item {\style The advection step:} \label{itm:adv-step-num}
The first step~\eqref{eq:GN-dt-adv} consists of the shallow water equations and advection equations for the scalar functions $\w$ and $\sigma$.
As in~\cite{Bouchut04} we apply a classical finite volume scheme for the coupled system of equations for water depth $h$ and horizontal velocity $\u$ and obtain for the $(n+1)$st time step the explicit scheme
\begin{equation*}
\mat{r@{~}c@{~}l}{\ds
h^{n*}_k
&=&\ds
h^n_k
- \frac{\dtn}{\m_k}\sum_{f\in\Fd_k}\F_f^h\cdot\n_k^\kf\m_f,
\\
h^{n*} \u^{n*}
&=&\ds
h^n  \u^n 
- \frac{\dtn}{\m_k}\sum_{f\in\Fd_k}\F_f^{hu}\n_k^\kf\m_f+\Sd_k^n. 
}
\end{equation*}
Here $\left(\F_f^h\cdot\n_k^\kf,\F_f^{hu}\n_k^\kf\right)=\F\left(\left(h^n_k,h^n_k\u^n_k\right)^\top,\left(h^n_\kf,h^n_\kf\u^n_\kf\right)^\top\right)$ is computed with an approximate Godunov method. 
Godunov schemes are stable assuming the classical CFL condition on the time step
\begin{equation}\label{CFL}
\delta_t^n=\frac{\Ccfl \delta_k}{\lambda^n}
\qquad\textrm{with}\qquad
\delta_k \coloneqq \frac{\m_k}{\sum_{f\in\Fd_k}\m_f}
\ ,
\end{equation}
for given CFL parameter $0<\Ccfl\le1$ and $\lambda^n\coloneqq \lambda\left(h^n,h^n\u^n\right)$ denotes an approximation of the largest magnitude of the shallow water eigenvalues.
The latter depends on the choice of the approximate Godunov solver, see~\cite{Bouchut04}. 
For the simulations presented below the HLL solver is used. 
We also tested alternative numerical fluxes (Rusanov, Roe, Lax--Friedrichs, Lax--Wendroff, kinetic) without significant differences in the results. 
At each discrete time $t^n$ the time step $\delta_t^n$ is computed to satisfy the CFL condition \eqref{CFL} with $\Ccfl=1$. 

Furthermore, $\Sd_k^n$ is a discretization of the source term including the bathymetry and we choose the hydrostatic reconstruction as in~\cite{Audusse04}. 
The discrete initial functions $h^0_\star$ and $\u^0_\star$ are obtained as approximations of the initial functions $h^0\left(x\right)$ and $\u^0\left(x\right)$ of the Green--Naghdi model. 
Several classical schemes for the shallow water equations use ghost cells for which the values $h^n_\kg$ and $\u^n_\kg\cdot\n^{\kg}_\ki$ are required.
In the case of the shallow water equations with subcritical boundary conditions only one datum is given. 
The second one is recovered based on the fact that the Riemann invariant leaving the domain is constant. 
In addition, the tangential component $\u^n_\kg \cdot \t_\ki^\kg$, with $\t_\ki^{\kg}$ a tangential unit vector to the face, of the velocity is needed (at least) when the mass flux is incoming, i.e., if $\F_f^h\cdot\n_{\ki}^{\kg}<0$. 
Note that the mass flux can be computed first by using the water depth and normal velocity in the vicinity of the face. 
For more details we refer to~\cite[\S V.2.2]{Godlewski96}. 
Note that in the case of the Green--Naghdi equations, there is no object corresponding to the Riemann invariant and hence a~priori such a strategy is not available. 

As second part of the advection step the vertical velocity as well as the tangent component of the horizontal velocity are computed using an upwind scheme with respect to the mass flux.
More precisely, we set
\begin{equation*}
\mat{r@{~}c@{~}l}{\ds
h^{n*}_k
\begin{pmatrix}\w^{n*}_k \\
\sigma^{n*}_k \end{pmatrix}
&=&\ds
h^{n}_k
\begin{pmatrix}
\w^{n}_k\\
\sigma^{n}_k \end{pmatrix}
- \frac{\dtn}{\m_k}\sum_{f\in\Fd_k}\left(
\begin{pmatrix}\w^{n}_k \\
\sigma^{n}_k \end{pmatrix}
\left(\F_f^h\cdot\n_k^\kf\right)_+-
\begin{pmatrix} \w^{n}_\kf \\
\sigma^{n}_\kf \end{pmatrix}
\left(\F_f^h\cdot\n_k^\kf\right)_-\right)\m_f,
}
\end{equation*}
with the positive and negative parts of a scalar $a$ defined by $(a)_\pm=\frac{1}{2}\left(\abs{a}\pm a\right)$. 
The discrete initial data $\w^0_\star$ and $\sigma^0_\star$ can be obtained from the discrete initial data $h^0_\star$ and $\u^0_\star$ and the ghost cell values $\u^0_\kg\cdot\n_\ki^\kg$ using the constraints of the discrete admissible set of function $\Ahd$. 
Again, for the tangential component of the horizontal velocity the values in the ghost cells $\w^n_\kg$ and $\sigma^n_\kg$ are required if the flow is incoming, i.e., if $\F_f^h\cdot\n_\ki^\kg<0$. 
However, assuming that the constraints in $\Ahd$ are satisfied in the ghost cells, only one of the two values $\w^n_\kg$ and $\sigma^n_\kg$ has to be imposed. 
Then, the remaining one is determined  by the value of the first and the gradient of the bathymetry. 
For the simple scheme present in~\secref{NumSimple} one can check that the ghost cell values are given by
\begin{equation*}
\sigma_\kg^n= \frac{1}{\sqrt{3}} \left(\w^n_\kg -\u^n_\kg \cdot \nablad_\kg B_\star \right).
\end{equation*}
\item {\style The correction step:} \label{itm:cor-step-num}
In the
correction step of the $(n+1)$st time step for the boundary condition we require the normal velocity $\u^{n+1}_\kg\cdot\n_\ki^\kg$ and the mean hydrodynamic pressure $h^{n+1}_\kg\q^{n+1}_\kg$ for the respective boundary faces, cf.~\secref{NumBounded}.
\end{enumerate}

Overall the scheme requires the ghost cell values
\begin{align*}
&(\mathrm{I}) \qquad \begin{cases}
h^n_\kg
\;\; \textrm{and}\;\; 
\u^n_\kg\cdot\n_\ki^\kg,
&
\\
\u^n_\kg\cdot\t_\ki^\kg
\;\; \textrm{and}\;\; 
\w_\kg^n \qquad \qquad 
&\text{if} \quad
\F_f^h\cdot\n_\ki^\kg<0,
\end{cases}
\\
&(\mathrm{II}) \qquad 
\u^{n+1}_\kg\cdot\n_\ki^\kg
\;\;\textrm{and}\;\;
h^{n+1}_\kg\q^{n+1}_\kg.
\end{align*}
However, thanks to the projection conditions~\eqref{NumProjCond}, some of them can be deduced from the others.
In the remaining part of this section let us describe some strategies suitable in practice. 
Note that the following list is merely a starting point for further investigation and meant to demonstrate the potential of our approach. 
 \medskip

Let us remark that the widely used choice of periodic boundary conditions is  included in the framework presented here. 
Choosing the ghost cell values as periodic extension, in~\eqref{eq:discr-proj} the faces in the periodic boundary can be treated in the same way as interior faces. Then, only the boundary faces in the remaining part of the boundary have to satisfy~\eqref{NumProjCond}.
\medskip

In the following, we propose strategies to fix boundary conditions for numerical schemes for the Green--Naghdi equations.
Note that there is no theoretical result available specifying the number of boundary values required.  
Even in the simpler case of the linear KdV equations such an analysis is already rather involved, cf.~\cite{A.2012}. 
Here we choose to impose $d+2$ values at the boundary if the flow is incoming and 2 values otherwise.
Despite the lack of theoretical justification this number seems to be consistent with the analysis of the steady state regime.
The strategies we present heavily rely on the choice of the prescribed unknowns. 
In~\secref{BC_H_U} the water depth and the normal velocity are prescribed, whereas in~\secref{BC_M_Q} the flow discharge and the hydrodynamic pressure are prescribed.
Eventually, in~\secref{BC_dry} we comment on the case of non-flat bathymetry and dry front conditions.

For all simulations the gravitational acceleration $g$ is set to $9.81$. 
The 1D computational domain is chosen as $\Omega=(0,1)$ and we use a uniform grid with spatial grid size $m_k=\delta_x$.
There are only two boundary faces and we denote them by $f=L$ for the left one at $x=0$ and $f=R$ for the right one at $x=1$.

\subsection{Fixed water depth and velocity boundary condition}\label{BC_H_U}

In this section we consider the case of prescribed water depth and normal velocity at the boundary.
More specifically, for given discrete real-valued boundary data $\bch_f^n>0$ and $\bcu_f^n$ we set
\begin{subequations}\label{eq:bc-hu}
\begin{equation}\label{eq:bc_h_u}
h_\kg^n = 2\bch_f^n-h_\ki^n
\qquad\textrm{and}\qquad
\u_\kg^n\cdot \n_\ki^\kg = 2\bcu_f^n-\u_\ki^n\cdot \n_\ki^\kg.
\end{equation}
The condition on the normal velocity results from setting $\alpha_f=0.5$ in~\eqref{NumBcInhom}. 
This choice is motivated by the conservation of mass when considering the wall boundary condition, see~\secref{BC_wall}.
For the faces with incoming flow for given discrete real-valued boundary data $\bcv_f^n$ and $\bcw_f^n$, we additionally set
\begin{equation}\label{BC_upwind}
\u_\kg^n\cdot \t_\ki^\kg
=
2\bcv_f^n-\u_\ki^n\cdot \t_\ki^\kg
\qquad\textrm{and}\qquad
\w_\kg^n
=
2\bcw_f^n-\w_\ki^n
\qquad\textrm{where}\quad
\F_f^h\cdot\n_\ki^\kg<0.
\end{equation}
The ghost cell values in~\eqref{eq:bc_h_u}, \eqref{BC_upwind} are exactly what is required for the advection step, see~\secref{sec:full-num}~\ref{itm:adv-step-num}. 
In the correction step \ref{itm:cor-step-num} the normal velocity is imposed and hence the faces in question are part of $\Gammad_u$. 
Then, the boundary condition on the hydrodynamic pressure is determined by~\eqref{NumBc} on $\Gammad_u$ and we set 
\begin{equation}\label{eq:bc-h-u-cor}
\u_\kg^{n+1}\cdot \n_\ki^\kg
=
2\bcu_f^{n+1}-\u_\ki^{n+1}\cdot \n_\ki^\kg
\qquad\textrm{and}\qquad
h_\kg^{n+1}\q_\kg^{n+1}
=h_\ki^{n+1}\q_\ki^{n+1}.
\end{equation}
\end{subequations}

\subsubsection*{Solitary waves}\label{BC_H_U_evidence}\def\qsol{\q_{\sol}}

We illustrate the efficiency of our strategy on solitary wave solutions of the Green--Naghdi equations in 1D. 
The advantage of this example is that an analytical solution is available and it can easily be compared to the numerical solution.
More precisely, for a given far field water level $H>0$, amplitude $A>0$ and initial position $X\in\R$ of the wave the solitary wave is given by
\begin{subequations}\label{soliton}
\begin{equation}\label{soliton1}
\mat{rr@{~}c@{~}l}{&\ds
\hsol\left(t,x\right)
&=&\ds
H\left(1+A\sech^2\left(\chi\left(t,x\right)\right)\right),
\\
%\textrm{and}
&\ds
\usol\left(t,x\right)
&=&\ds
\left(1-\frac{H}{\hsol\left(t,x\right)}\right)\sqrt{\left(1+A\right)gH},
}
\end{equation}
with 
${\chi\left(t,x\right)
=
\sqrt{\frac{3A}{4\left(1+A\right)}}
\frac{1}{H} \left(x-X-\sqrt{\left(1+A\right)gH}t\right)}$.
Using the constraints~\eqref{eq:GNconstraints} and the equations~\eqref{eq:GN-hw}, \eqref{eq:GN-hs} of the Green--Naghdi equations, the remaining velocity functions and the hydrodynamic pressure functions are determined by
\begin{equation}\label{soliton2}
\mat{rr@{~}c@{~}l}{&\ds
\wsol\left(t,x\right)
&=&\ds
\sqrt{\frac{3g}{4}}\frac{\left(AH\right)^{3/2}\sech^2\left(\chi\left(t,x\right)\right)\tanh\left(\chi\left(t,x\right)\right)}{\hsol\left(t,x\right)},
\\&\ds
\ssol\left(t,x\right)
&=&\ds 
\frac{1}{\sqrt{3}}\wsol\left(t,x\right),
\\
&\ds
\qsol\left(t,x\right)
&=&\ds
g\left(\frac{3+A}2H-\hsol\right)
-\frac{1}{2}\left(\left(\usol-\sqrt{(1+A)gH}\right)^2+\wsol^2+\ssol^2 \right),
\\
%\textrm{and}
&\ds
\qBsol\left(t,x\right)
&=&\ds
\frac32\qsol\left(t,x\right).
}\end{equation}
\end{subequations}

\paragraph*{Numerical evidence (A).}
We consider the solitary wave~\eqref{soliton} for the parameters $A=0.5$ and $X=-0.5$ and compute the numerical solution starting from the corresponding initial data. 
We compare the solutions for $t \leq \sqrt{\frac{1}{(1+A)gH}}$, i.e., until the wave reaches the center of the computational domain.  
We investigate two cases: for $H=10^{-1}$ the solitary wave is a faster, rather \textit{spread wave}, and for $H=10^{-2}$ it is a slower, more \textit{stiff wave}. 
The initial functions are $h^0_k\coloneqq \hsol \left(0,\left(k-0.5\right)\delta_x\right)>0$ and $\u^0_k\coloneqq\usol\left(0,\left(k-0.5\right)\delta_x\right)$, which means that initially the water surface is almost flat in the computational domain.
The initial vertical velocities $\w^0_k$ and $\sigma^0_k$ are computed using the discrete constraints~\eqref{NumAdmissibleSet} and hence the vector of initial velocities is admissible in the sense of $\mathbb{A}^{\delta}_{h_\star^0}$.
For this set of simulations we prescribe the water depth and the velocity on both boundary faces $L$ and $R$, see~\eqref{eq:bc-hu}. 
To be precise, we set $\bch_L^{n}\coloneqq \hsol\left(t^{n},0\right)$ and $\bcu_L^{n}\coloneqq \usol\left(t^{n},0\right)$ on the left boundary face and $\bch_R^{n}\coloneqq \hsol\left(t^{n},1\right)$ and $\bcu_R^{n}\coloneqq\usol\left(t^{n},1\right)$ on the right boundary face.
In addition, on the left the flow is incoming and we set $\bcw_L^{n}\coloneqq\wsol\left(t^{n},0\right)$.
The numerical solutions are computed for several spatial mesh sizes $\delta_x$ between $10^{-2}$ and $10^{-4}$.

\begin{figure}[ht]
\centering
\includegraphics[width=0.95\linewidth,clip,trim=0 41 0 13]{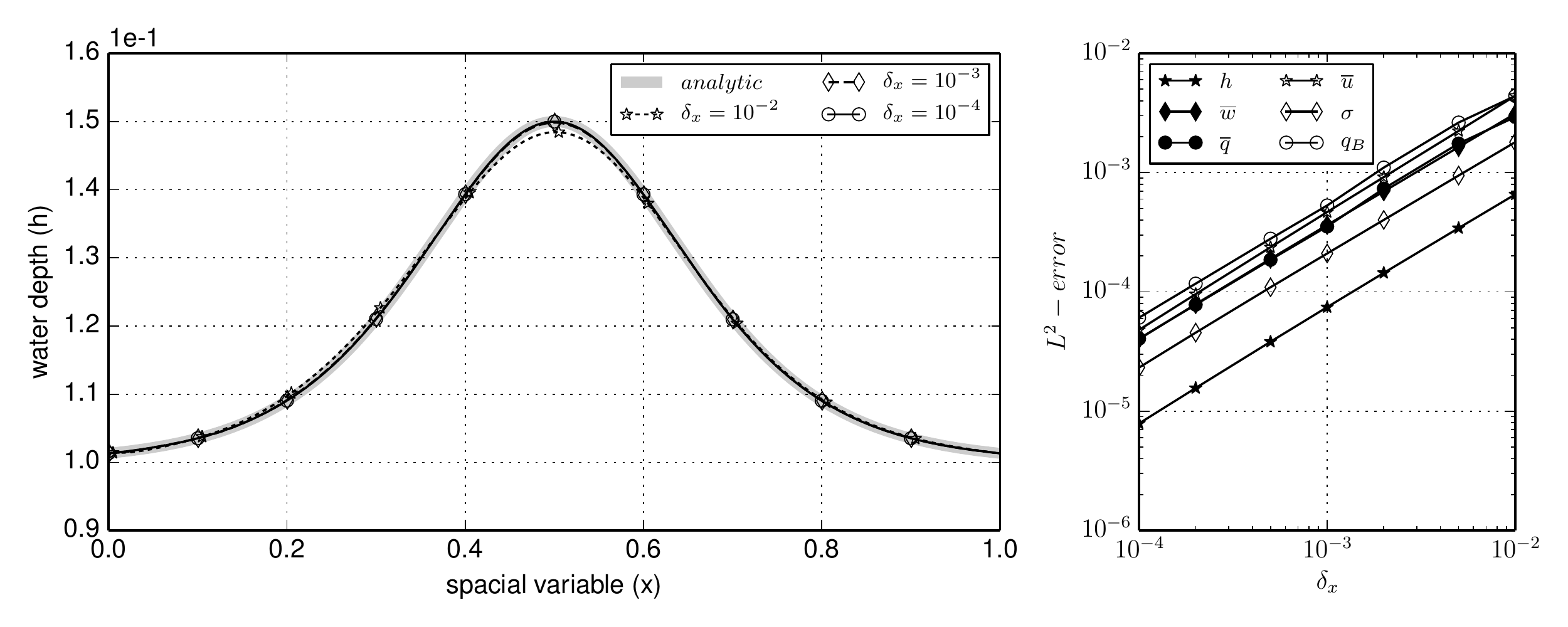}
\includegraphics[width=0.95\linewidth,clip,trim=0 13 0 13]{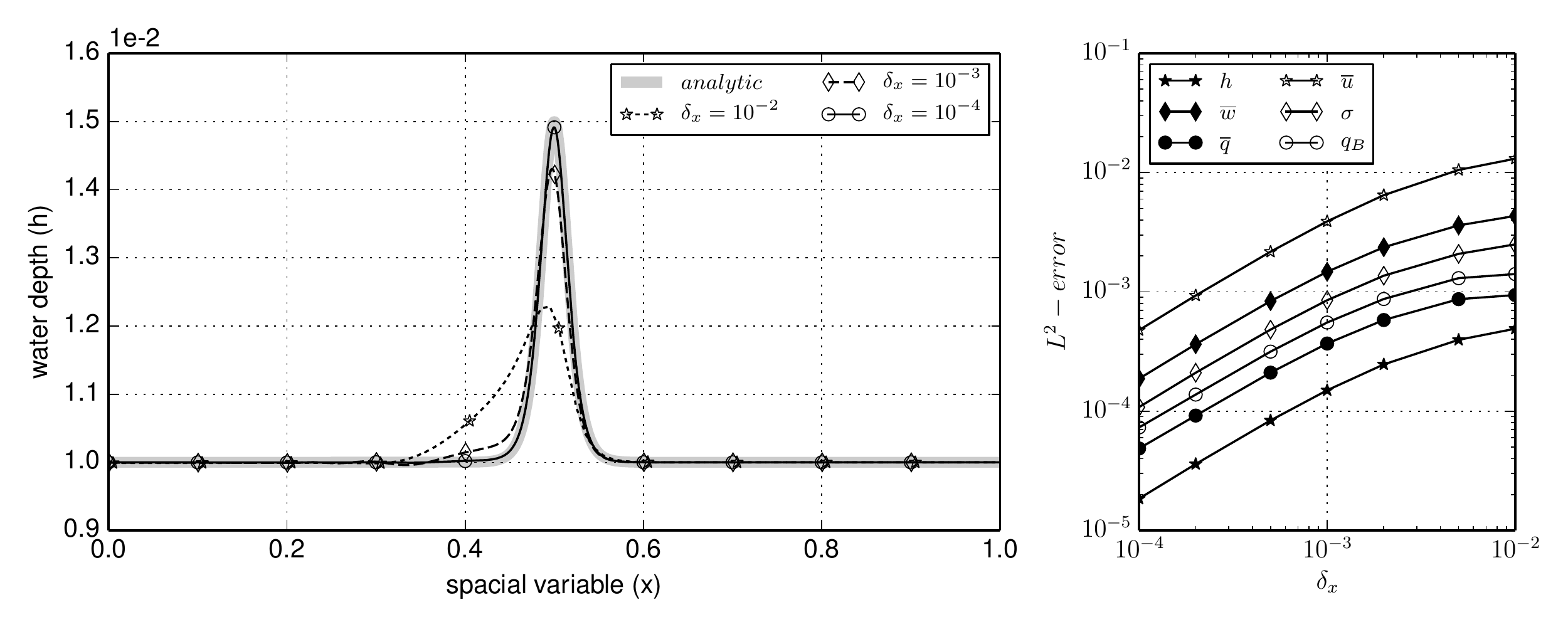}
\caption{\label{fig:BC_H_U} 
Numerical evidence (A): Exact and approximate solutions for initial soliton~\eqref{soliton1}  with parameter $A=0.5$, $X=-0.5$ and $H=10^{-1}$ (upper) or $H=10^{-2}$ (lower) at time $t=\sqrt{\frac{1}{(1+A)gH}}$.
Water levels (left) and spatial $L^2$-error as function of the mesh size (right).}
\end{figure}

\figref{fig:BC_H_U} shows the numerical solutions and the exact solution when the wave has traveled to the center of the computational domain.
The upper plots visualize the simulations of the spread wave for $H=10^{-1}$, whereas the plots  below contain the simulations of the stiff wave for $H=10^{-2}$. 
The water levels are plotted on the left.
The plots on the right-hand side show the $L^2$-errors between the computed solution and the analytical solution of each unknown, given by
\begin{align*}
\left(\delta_x\sum_{k\in\T}\l|\phi_k^n-\phisol\left(t^n,\left(k-0.5\right)\delta_x\right)\r|^2\right)^{1/2}  \quad \text{ for } t^n = t. 
\end{align*}

We observe that the wave profile is well recovered for the spread wave for $H=10^{-1}$, even though for a coarser meshes some numerical diffusion can be observed, see \figref{fig:BC_H_U} top left. 
The convergence rate of the $L^2$-errors is first order as expected, see \figref{fig:BC_H_U} top right. 
In the case of the more localized wave for $H=10^{-2}$, the numerical diffusion is larger, see \figref{fig:BC_H_U} bottom left. 
Still, the same amount of numerical diffusion is observed for the simulation of a steady soliton, see~\cite{Parisot19}. 
This means that the effect is not due to the boundary conditions and hence it confirms that the boundary conditions are suitably chosen. 
The first order convergence rate is recovered for sufficiently small mesh sizes, see \figref{fig:BC_H_U} bottom right.

\subsubsection{First variant: the wall boundary condition}\label{BC_wall}
The widely used wall boundary condition is a special case of the boundary conditions described in~\eqref{eq:bc-hu}. 
Indeed, the wall boundary condition is recovered by setting $\bch_f^n=h_\ki^n$ and $\bcu_f^n=0$. 
Thanks to the centered reconstruction with $\alpha_f=0.5$, the ghost cell values are given by 
\begin{equation*}
h_\kg^n=h_\ki^n
\qquad\textrm{and}\qquad
\u_\kg^n\cdot\n_\ki^\kg=-\u_\ki^n\cdot\n_\ki^\kg.
\end{equation*}
For any numerical solver satisfying the Galilean invariance this leads to a vanishing mass flux $\F_f^h\cdot\n_\ki^\kg=0$, and hence no additional data is required for the advection step.
For the correction step by~\eqref{eq:bc-h-u-cor} the condition on the hydrostatic pressure reads
\begin{equation*}
h_\kg^{n+1}\q_\kg^{n+1}
=h_\ki^{n+1}\q_\ki^{n+1}.
\end{equation*}
Several simulations with wall boundary condition have been presented in~\cite{Parisot19}. 

\subsubsection{Second variant: the transparent condition}\label{BC_H_U_transparent}

A further application of boundary conditions of the form~\eqref{eq:bc-hu} consists in mimicking transparent boundary conditions. 
Transparent boundary conditions are used when waves are supposed to leave the computation domain without being affected by the boundary, e.g., by reflections. 
A numerical scheme for the Green--Naghdi model subject to transparent boundary condition is proposed in~\cite{KN.2020}. 
Here we suggest an alternative, which we expect to be simpler to implement at the cost of being less accurate with respect to reflections at the boundary. 

For the shallow water model a simple strategy consists in imposing Neumann boundary conditions for both water depth and normal velocity
\begin{equation*}
\bch^{n+1}_f=h_\ki^{n*}=h_\ki^{n+1}
\qquad\textrm{and}\qquad
\bcu^{n+1}_f=\u_\ki^{n*}\cdot\n_\ki^\kg.
\end{equation*}
Also higher order Neumann boundary condition can be considered, see~\cite{Coulombel20}. 
 Conceptually, transparent boundary conditions are of interest on parts of the boundary where outflow occurs. 
If, however, the flow is incoming, then one has to specify the remaining components of the velocity by~\eqref{BC_upwind}, for example by setting $\bcv_f^n=0$ and $\bcw_f^n=0$. 
The conditions on the hydrostatic pressure remain as in~\eqref{eq:bc-h-u-cor}. 

\paragraph*{Numerical evidence (B).}
To illustrate the numerical strategy we compare the computed solution with the analytical solitary wave~\eqref{soliton} with the parameters $H=10^{-2}$, $A=0.5$ and $X=0.5$. 
We choose the initial data in a way that the center of the wave coincides with the center of the computational domain initially. 
Then, we simulate the propagation until the wave has left the computational domain at $t=\sqrt{\frac{1}{(1+A)gH}}$. 
The numerical solutions are computed for several spatial mesh sizes $\delta_x$ between $10^{-2}$ and $10^{-4}$.
We use the transparent boundary condition as described above on both boundary faces $L$ and $R$. 

\begin{figure}[ht]
\centering
\includegraphics[width=0.97\linewidth,clip,trim=0 28 0 12]{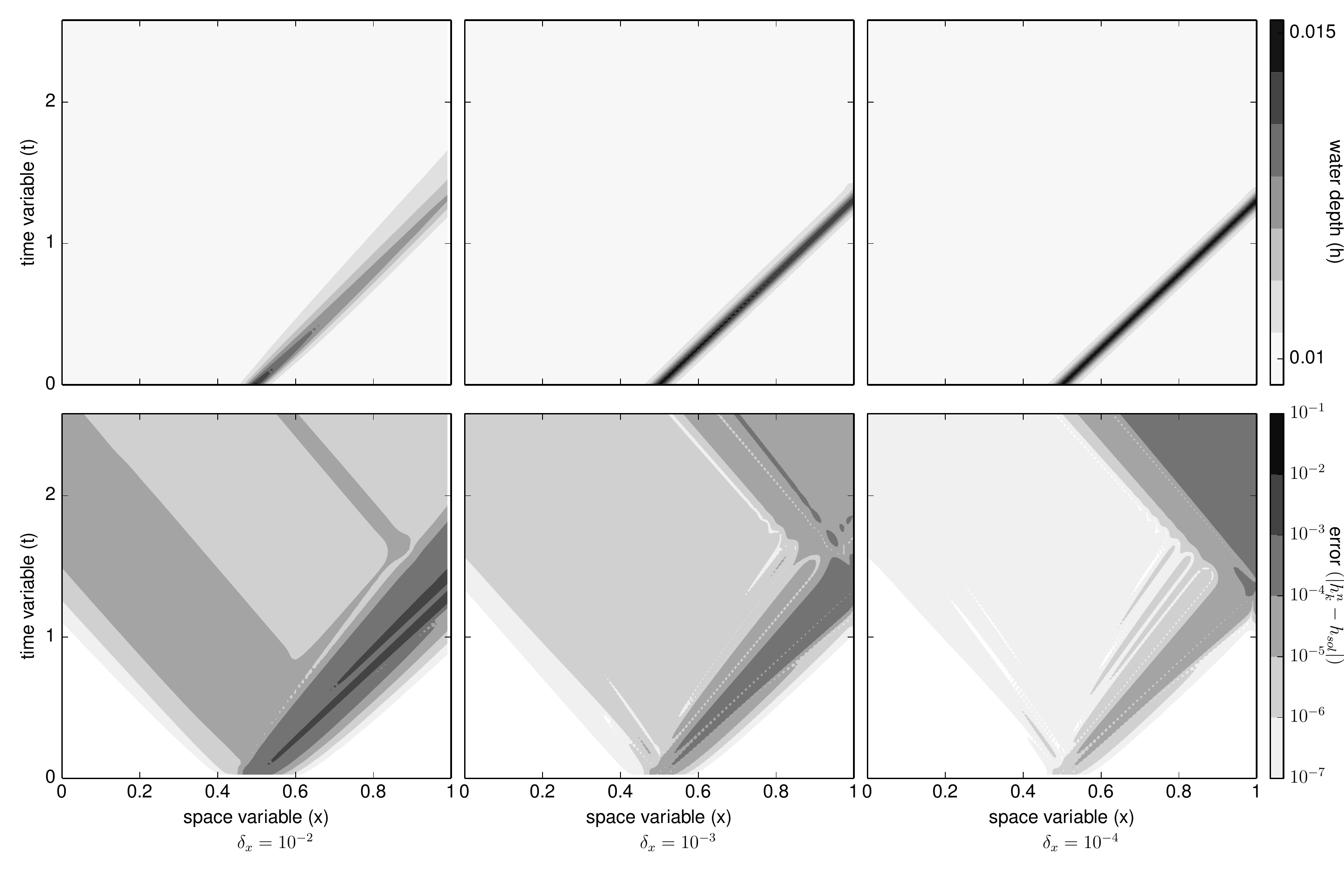}
\caption{\label{fig:BC_W_T_1d-2}
Numerical evidence (B): 
Simulations for initial soliton~\eqref{soliton1} with parameter $A=0.5$, $X=0.5$ and $H=10^{-2}$. 
Intensity graphs of the water depth (top) and the logarithmic error of the water depth (bottom) on the $(x,t)$-plane for spatial mesh size $\delta_x=10^{-2}$ (left), $\delta_x=10^{-3}$ (middle) and $\delta_x=10^{-4}$ (right).
}
\end{figure}

The results are presented in \figref{fig:BC_W_T_1d-2} as intensity charts over the $\left(x,t\right)$ plane. 
The upper row shows the water depth and the lower one the error of the water depth with logarithmic scaling. 
The columns correspond to different choices of resolution with $\delta_x=10^{-2}$ in the left, $\delta_x=10^{-3}$ in the middle and $\delta_x=10^{-4}$ in the right column.

 We note that the behavior of the solution is globally well recovered for each of the mesh parameters, in the sense that the wave leaves the domain without significant perturbation or numerical instability, see \figref{fig:BC_W_T_1d-2} top. 
The error plots (bottom) allow for a more detailed comparison. 
Clearly the left boundary condition works well.
 This is to be expected because the initial function is sufficiently flat at the left boundary. 
But also when the flow is incoming the transparent boundary seems to be robust.  
On the right boundary the flow is outgoing and we observe small reflections.
For coarse meshes with $\delta_x=10^{-2}$, the numerical diffusion in the domain has a larger effect than the reflection wave. 
For fine meshes with $\delta_x=10^{-4}$ on the other hand, the effect of the reflection wave is larger than the numerical diffusion. 
Still, the error due to reflection does not exceed the order of the resolution $O\left(\delta_x\right)$. 
Note that the error on the right boundary increases with decreasing $\delta_x$.
This is due to the fact that the amplitude of the reflection wave is usually proportional to the amplitude of the approximate gradient of the leaving wave.
When the spatial mesh size decreases, then the gradient of the leaving wave increases since the numerical diffusion is lower.
Similar effects can be observed for the shallow water model when a similar strategy for the transparent boundary condition is applied. 

\subsection{Fixed discharge and hydrodynamic pressure}\label{BC_M_Q}

In this section we propose to prescribe the discharge $h \u \cdot \nu$ and the hydrodynamic pressure $h \q$ at the boundary under the assumption that $\u_\ki^{n}\cdot \n_\ki^\kg \neq 0$ on the interior cell. 
For given discrete real-valued data $\bchu_f^n$ and $\bchq_f^n$, we set
\begin{subequations}\label{eq:bc-disc-pres}
\begin{equation*}
h_\kg^n\u_\kg^n\cdot \n_\ki^\kg
=
2\bchu_f^n-h_\ki^n\u_\ki^n\cdot \n_\ki^\kg
\qquad\textrm{and}\qquad
h_\kg^n\q_\kg^n
=
2\bchq_f^n-h_\ki^n\q_\ki^n,
\end{equation*}
on the respective boundary faces. 
This does not yield all the ghost cell values required in the advection step.
However, since the hydrodynamic pressure is prescribed the respective faces are contained in $\Gammad_{hq}$. Thus, by~\eqref{NumBc} for the normal velocity we have that 
\begin{equation*}
\u_\kg^{n}\cdot \n_\ki^\kg
=
\u_\ki^{n}\cdot \n_\ki^\kg.
\end{equation*}
In combination with the condition on the discharge, the water depth is given by
\begin{equation*}
h_\kg^n
=
2\frac{\bchu_f^n}{\u_\ki^{n}\cdot \n_\ki^\kg}-h_\ki^n,
\end{equation*}
since $\u_\ki^{n}\cdot \n_\ki^\kg \neq 0$. 
Note that the condition on $h^n_{k_g}$ is not linear.
Since $h^n_{k_g}$ is not needed in the correction step, but only in the explicit advection step, this does not cause any issues. 
Also note that the water depth in the ghost cell is positive provided that
\begin{equation*}
\sign\left(\bchu_f^n\right)=\sign\left(\u_\ki^{n}\cdot \n_\ki^\kg\right)
\qquad\textrm{and}\qquad
\l|\bchu_f^n\r|>\frac12h_\ki^n\l|\u_\ki^{n}\cdot \n_\ki^\kg\r|.
\end{equation*}
\end{subequations}
For sufficiently regular discharge at the boundary this assumption is satisfied.

\paragraph*{Numerical evidence (C).}
It is well known that there exist non-trivial space-periodic steady solutions to the Green--Naghdi model, the so-called cnoidal waves, see~\cite{Cienfuegos11}.
By fixing the water depth and the velocity as for the strategy proposed in~\secref{BC_H_U}, with stationary boundary data and flat bottom it is not possible to obtain solutions other than the trivial solutions.
In contrast, the strategy described in~\secref{BC_M_Q} offers this possibility. 

Let us describe the test case in detail. 
We initialize the flow with a flat free surface with $h^0\left(x\right)=10^{-1}$ and a constant horizontal velocity $\u^0\left(x\right)=5\cdot10^{-1}$. 

On the right boundary face $R$ we prescribe boundary values according to the strategy presented in~\secref{BC_H_U}, setting $\bch_R=10^{-1}$ and $\bcu_R=5\cdot10^{-1}$. 
Since the flow is outgoing no additional values are required.
At the left boundary face $L$ we prescribe the discharge and the hydrodynamic pressure using the strategy presented above in \secref{BC_M_Q} with $\bchu_L=5\cdot10^{-2}$ and several values for $\bchq_L$. 
Since on the left boundary the flow is incoming we have to prescribe also the vertical velocity and we set $\bcw_L=0$.

\begin{figure}
\centering
\includegraphics[width=0.97\linewidth,clip,trim=0 13 0 13]{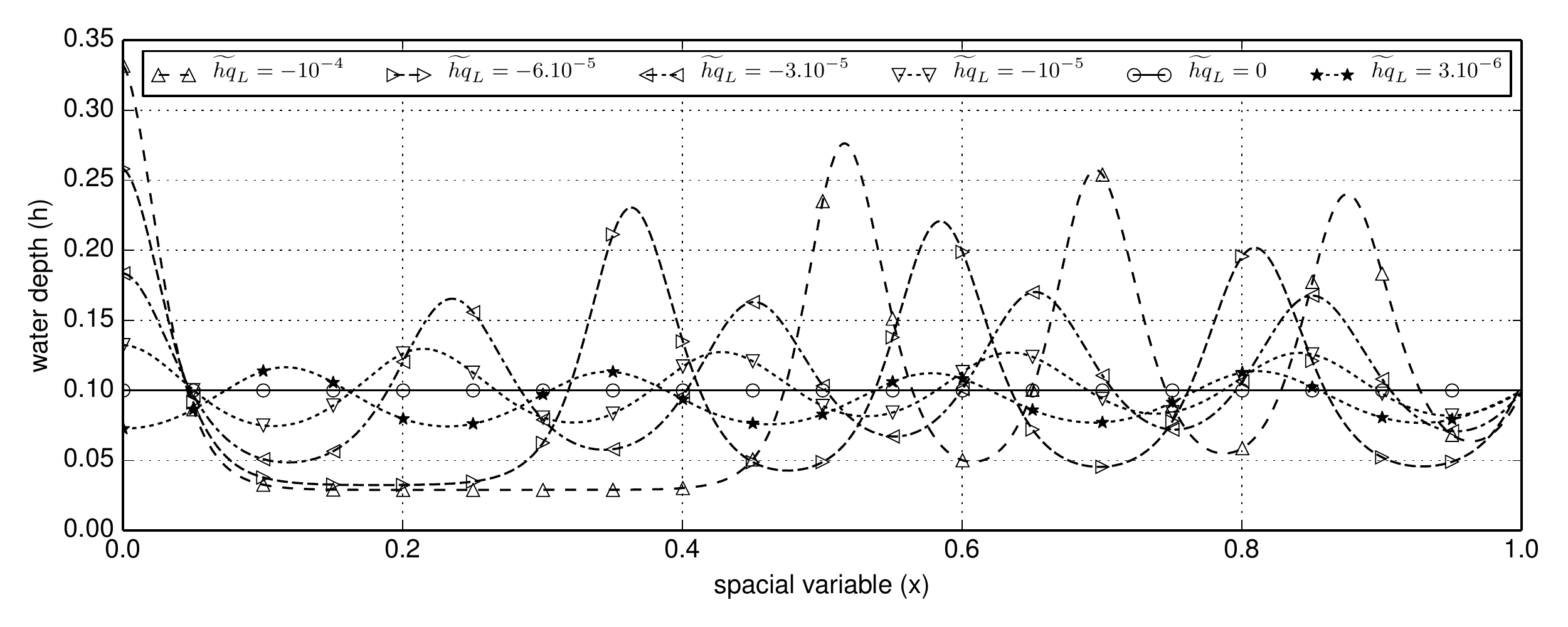}
\caption{\label{fig:BC_M_Q}
Numerical evidence (C):
Simulations for imposed discharge and hydrodynamic pressure at the inlet $L$ for several values of $\bchq_L$.  Water depth at time $t=10$.
}
\end{figure}

\figref{fig:BC_M_Q} shows the water depth of the approximate solutions at time $t=10$ computed with $\delta_x=5\cdot10^{-4}$ for several values of the hydrodynamic pressure at the left boundary. 
Note that these solutions are not steady state solutions and in fact we do not know whether there exist a steady solution satisfying the same boundary conditions.
Still, let us compare our approximate solution to what is known about the steady cnoidal wave solutions in a qualitative manner. 

The steady solutions are periodic, which seems to be the case also for our simulations provided that the value imposed on the  hydrodynamic pressure on the left boundary is sufficiently small with $\abs{\bchq_L}\le3\cdot10^{-5}$.
For larger values imposed on the hydrodynamic pressure at the left boundary the amplitude of the wave decreases in $x$ due to numerical diffusion. 
For $\bchq_L\le-6\cdot10^{-5}$ the solution has a low plateau which has already been observed in~\cite{Cienfuegos11}.

Let us comment on the behavior of the approximate solutions close to the left boundary.
For a sufficiently smooth steady solution by constraint~\eqref{eq:GNconstraints} the gradient of the water depth is linked to the value of the imposed vertical velocity $\bcw_L$. 
In our case $\bcw_L=0$ we conclude that the free surface has to be flat at the left boundary, which matches with the simulations. 
Furthermore, one can link the hydrodynamic pressure to the second derivative of the water depth function. 
More precisely, at the extreme point of the water depth function the second derivative of the water depth function is proportional to the hydrodynamic pressure for the steady state. 
Consistently, we observe that if $\bchq_L<0$, then the water depth function has a maximum at the left boundary and vice versa.

\subsection{Simulation on a beach}\label{BC_dry}

The last type of boundary investigated in this article is the dry front. 
This is a special type of boundary because it is   not a boundary in the advection step but only in the correction step.
As explained before, this boundary is detected after the advection step and treated by setting $\u_k^{n+1}=0$ and $h_k^{n+1}\q_k^{n+1}=0$ in any cell $k  \in \mathbb{T}$ with $h_k^n=0$. 
This strategy was already employed and numerical evidence was presented in~\cite[\S4.4]{Parisot19}.
Therein, the dry front is treated with a zero velocity based on the argument that this choice ensures entropy stability.
Here, we establish the more fundamental discrete projection property, which implies entropy stability by construction. 
Hence, the numerical treatment of the dry front is the same in both cases and we shall refrain from repeating it.

\section{Conclusion}\label{sec:conclusion}

We have characterized and investigated a class of boundary conditions for the time-discrete Green--Naghdi equations with bathymetry. 
Several boundary conditions relevant for practical studies such as dry areas, transparent boundary conditions and wave generating boundary conditions can be formulated in this setting. 
The types of boundary conditions we propose are designed in a way that for the correction step the projection property is preserved for bounded domains and homogeneous boundary conditions, as is standard for incompressible fluid equations. Furthermore, for inhomogeneous boundary conditions the resulting equations in the correction step are independent of any reference functions.
We obtain well-posedness for the time-discrete and the fully discrete correction step and the underlying 
entropy inequality is verified. 
The latter ensures robustness of the numerical scheme and the approach is expected to extend to higher order schemes. 
Note that for the sake of simplicity we have presented the most basic collocated scheme in the fully discrete case. 
Indeed, the objective is to illustrate the benefits of our general approach rather than to design the best possible numerical scheme within this framework. 
However, with this choice we have not aimed for a discrete inf-sup condition independent of $\delta_t^n$. 
Indeed, the factor $\delta_t^n$ may cause spurious modes to appear in the pressure functions for small time steps, see, e.g.,~\cite{Guermond06,GQ.1998} for a discussion for other projection methods. 
But the accurate recovery of the pressure is not our main focus. 
In fact the pressure functions do not have to be computed in each time step since they do not enter the following time step because our scheme is non-incremental. 
In frequently used formulations of the Green--Naghdi equations the pressure functions do not even appear and hence seem to be of lesser interest. 
Note also that the investigation of the discrete inf-sup condition is hampered by the fact that stationary solutions are not well-understood yet and the inner product changes with $h$ in every time step. 

This leads us to several theoretical and numerical open questions. 
Firstly, recall that our analysis starts from the time discrete model. 
For the fully continuous equations it is not clear that a certain projection property is available, since the water depth can not be considered as a parameter in this case and the constraints are nonlinear. 
Investigating convergence of approximate solutions of the time-discrete problem as $\delta_t^n \to 0$  would not only justify the semi-discretization. 
Also it might allow for first insights into the boundary conditions for the full system of equations, which is currently an open problem. 
Furthermore, some technical issues are beyond the scope of this work such as dry areas in the space-continuous setting and the problem of moving bathymetry at the boundary. 
Regarding the first, one would have to work with weighted Sobolev spaces and assume or prove sufficient regularity for the degenerating water depth. 
For the latter, even in the time-discrete case the constraints are only affine, and hence the choice of suitable reference functions is less straightforward. 

Concerning the fully discrete problem also the investigation of the stationary solutions of the Green--Naghdi equations is of interest. 
In particular, the projection structure represents a promising starting point for the design of well-balanced schemes. 
Furthermore, since the framework seems to allow for a broad range of solutions, we are one step closer to the investigation of solutions for which the energy is not conserved.
From a more practical point of view also certain adaptive strategies are within reach, for which the projection step is only computed in those parts of the domain where it is really needed.
However, the implementation of such a local strategy requires a range of numerical parameters and the optimal choice does not seem to be obvious. 

Last but not least, let us remark that the notion of solutions that preserves the projection structure in bounded domains is most likely not the only sensible one.
It might be possible to work with alternative notions of solutions to the problem on bounded domains leading to well-posedness.

\paragraph*{Acknowledgements.}
We would like to express our sincere thanks to Thierry Gallou\"et for his valuable comments.
He greatly helped us understand some intricate points related to the weak formulation of the boundary conditions. 
This work was initialized during a stay of M. Parisot at the Aachen University founded by the DAAD program and the RWTH Aachen University. 
S. Noelle and T. Tscherpel were partially funded by the DFG project GRK2326~'Energy, Entropy, and Dissipative Dynamics (EDDy)'.

\bibliographystyle{acm}
\bibliography{lib-GN-bd}

\end{document}